\documentclass[12pt]{amsart}

\usepackage[top=1in, bottom=1in, left=1in, right=1in]{geometry}
\usepackage{times}

\usepackage{amssymb,amsmath,amsthm,mathtools,bbm}
\usepackage{enumitem}

\setlist[1]{itemsep=0.5em, topsep=0.5em}

\usepackage{color}
\definecolor{red}{rgb}{1,0,0}
\definecolor{orange}{rgb}{0.7,0.3,0}
\definecolor{blue}{rgb}{0,.3,.7}
\definecolor{green}{rgb}{0,.6,.4}
\usepackage[colorlinks=true, linkcolor=blue, citecolor=green, urlcolor=purple]{hyperref}   

\numberwithin{equation}{section}

\theoremstyle{plain}
\newtheorem{theorem}{Theorem}

\newtheorem{lemma}{Lemma}[section]
\newtheorem{cororollary}[lemma]{Corollary}
\newtheorem{proposition}[lemma]{Proposition}

\theoremstyle{remark}

\newtheorem*{rem*}{Remark}

\theoremstyle{definition}
\newtheorem{dfn}{Definition}

\newcommand{\N}{\mathbb{N}}
\newcommand{\Z}{\mathbb{Z}}

\newcommand{\R}{\mathbb{R}}
\newcommand{\C}{\mathbb{C}}

\newcommand{\CA}{\mathcal{A}}
\newcommand{\CB}{\mathcal{B}}

\newcommand{\CE}{\mathcal{E}}

\newcommand{\CK}{\mathcal{K}}
\newcommand{\CL}{\mathcal{L}}

\newcommand{\CN}{\mathcal{N}}

\newcommand{\CP}{\mathcal{P}}

\newcommand{\CR}{\mathcal{R}}
\newcommand{\CS}{\mathcal{S}}

\newcommand{\CV}{\mathcal{V}}
\newcommand{\CW}{\mathcal{W}}
\newcommand{\CX}{\mathcal{X}}

\newcommand{\al}[1]{\begin{align} #1 \end{align} }
\newcommand{\als}[1]{\begin{align*} #1 \end{align*} }

\newcommand{\ds}{\displaystyle}

\newcommand{\dee}{\,\mathrm{d}}

\newcommand{\lcm}{\operatorname{lcm}}

\newcommand{\eps}{\varepsilon}
\renewcommand{\epsilon}{\varepsilon}
\renewcommand{\phi}{\varphi}

\newcommand{\bs}\boldsymbol{}

\newcommand{\bg}{\big}
\newcommand{\bgg}{\Big}
\newcommand{\bggg}{\bigg}
\newcommand{\bgggg}{\Bigg}

\renewcommand{\le}{\leqslant}

\renewcommand{\ge}{\geqslant}
{}

\renewcommand{\tilde}{\widetilde}

\definecolor{red}{rgb}{1,0,0}

\definecolor{orange}{rgb}{0.7,0.3,0}

\definecolor{blue}{rgb}{.2,.6,.75}

\definecolor{green}{rgb}{.4,.7,.4}

\newcommand{\un}{\mathbbm{1}}

\begin{document}
	\title{An almost sharp quantitative version of the Duffin--Schaeffer conjecture}

	\author{Dimitris Koukoulopoulos}
	\address{D\'epartement de math\'ematiques et de statistique\\
		Universit\'e de Montr\'eal\\
		CP 6128 succ. Centre-Ville\\
		Montr\'eal, QC H3C 3J7\\
		Canada}
	\email{dimitris.koukoulopoulos@umontreal.ca}

	\author{James Maynard}
	\address{Mathematical Institute, Radcliffe Observatory quarter, Woodstock Road, Oxford OX2 6GG, England}
	\email{james.alexander.maynard@gmail.com}
	
		\author{Daodao Yang}
		\address{D\'epartement de math\'ematiques et de statistique\\
		Universit\'e de Montr\'eal\\
		CP 6128 succ. Centre-Ville\\
		Montr\'eal, QC H3C 3J7\\
		Canada}
		\email{yangdao2@126.com}

\subjclass[2010]{Primary: 11J83. Secondary: 05C40}
\keywords{Diophantine approximation, Metric Number Theory, Duffin--Schaeffer conjecture, graph theory, density increment, compression arguments}

\date{\today}

\begin{abstract} We prove a quantitative version of the Duffin--Schaeffer conjecture with an almost sharp error term. Precisely, let $\psi:\mathbb{N}\to[0,1/2]$ be a function such that the series $\sum_{q=1}^\infty \varphi(q)\psi(q)/q$ diverges. In addition, given $\alpha\in\mathbb{R}$ and $Q\geqslant1$, let $N(\alpha;Q)$ be the number of coprime pairs $(a,q)\in\mathbb{Z}\times\mathbb{N}$ with $q\leqslant Q$ and $|\alpha-a/q|<\psi(q)/q$. Lastly, let $\Psi(Q)=\sum_{q\leqslant Q}2\varphi(q)\psi(q)/q$, which is the expected value of $N(\alpha;Q)$  when $\alpha$  is   uniformly chosen from $[0, 1]$.  We prove that $N(\alpha;Q)=\Psi(Q)+O_{\alpha,\varepsilon}(\Psi(Q)^{1/2+\varepsilon})$ for almost all $\alpha$ (in the Lebesgue sense) and for every fixed $\varepsilon>0$. This improves upon results of Koukoulopoulos--Maynard and of Aistleitner--Borda--Hauke.
\end{abstract}

\maketitle
	
	\section{Introduction}
	
	A classical result of Khintchine \cite{Khinchine} states that if $\psi:\N\to[0,1/2]$ is a function such that $q\psi(q)$ is non-increasing and $\sum_{q}\psi(q)=\infty$, then for almost all $\alpha\in\R$ (that is, for $\alpha$ in a set of full Lebesgue measure) there are infinitely many pairs $(a,q)\in\Z\times \N$ such that
	\begin{equation}
		\label{eq:Khintchine}
	\bgg|\alpha-\frac{a}{q}\bgg| < \frac{\psi(q)}{q}.
	\end{equation}
	Conversely, if $\sum_{q}\psi(q)<\infty$ then the Borel--Cantelli lemma implies that for almost all $\alpha\in\R$ there are only finitely many pairs $(a,q)\in\Z\times \N$ satisfying \eqref{eq:Khintchine}. Hence, Khintchine's theorem is optimal up to the hypothesis that $q\psi(q)$ is non-increasing. In fact, in modern treatments of Khintchine's theorem (e.g., see \cite{Harman}) this condition is relaxed to simply require that $\psi$ is non-increasing.

	Erd\H os \cite{Erdos} and Schmidt \cite{Schmidt} proved a quantitative version of Khintchine's theorem: if $\psi$ is non-increasing and we let
	\[
	\Psi_1(Q):= 2\sum_{q\le Q}\psi(q),
	\]
	then we have that \cite[Theorem 4.1]{Harman} 
	\begin{equation}
	\#\Bigl\{(a,q)\in\mathbb{Z}^2:\,q\le Q,\,\eqref{eq:Khintchine} \text{ holds}\Bigr\}
	=	\Psi_1(Q)	+	O_{\eps}\Bigl(\Psi_1(Q)^{1/2} (\log{\Psi_1(Q)})^{2+\eps}\Bigr)
			\label{eq:QuantitativeKhintchine}
	\end{equation}
	for almost all $\alpha\in\R$, and for all $Q$ sufficiently large in terms of $\alpha$ and $\psi$. As a matter of fact, if we impose some even stronger conditions on $\psi$, we know from work of Fuchs \cite{Fuchs} that the quantity on the left-hand side of \eqref{eq:QuantitativeKhintchine} satisfies a Central Limit Theorem. In particular, the error term on the right-hand side of \eqref{eq:QuantitativeKhintchine} is nearly sharp.

	Duffin and Schaeffer \cite{DS} undertook a study to understand to what extent Khintchine's theorem holds without the hypothesis that $\psi$ is non-increasing. They realized that the fact that the fractions $a/q$ are not reduced can lead to degenerate situations, with the series $\sum_{q=1}^\infty \psi(q)$ diverging but with the inequality \eqref{eq:Khintchine} having only finitely many solutions for almost all $\alpha$. This led them to formulate the following conjecture: if $\psi:\N\rightarrow \R_{\ge 0}$ is a function such that $\sum_{q=1}^\infty \psi(q)\phi(q)/q=\infty$ (with $\phi$ denoting the Euler totient function), then for almost all $\alpha\in\R$ the inequality \eqref{eq:Khintchine} is satisfied by infinitely many pairs $(a,q)\in\Z\times\N$ of {\it coprime} integers. Conversely, the first Borel--Cantelli lemma readily implies that if $\sum_{q=1}^\infty \psi(q)\phi(q)/q<\infty$, then for almost all $\alpha\in\R$ there are only finitely many coprime pairs $(a,q)\in\Z\times\N$ satisfying \eqref{eq:Khintchine}. This would therefore give a full 0-1 law with a simple criterion to determine whether the measure of approximable $\alpha$ is full or null. 
	
	The Duffin--Schaeffer conjecture was resolved in full by the first two named authors \cite{KM} after important partial results by many authors \cite{aistleitner, aistleitner2, ALMTZ, DS-survey-article, extra-div2, hausdorff DS, catlin, dyson, erdos, erdoskorado, gallagher, extra-div1, PV, vaaler}. See also \cite{DS-exposition} for an exposition on the history of the conjecture.
	
	The work of \cite{KM} left open the quantitative question of the number of approximations, corresponding to \eqref{eq:QuantitativeKhintchine}. Given a function $\psi:\N\to[0,1/2]$, let
	\begin{equation}
	\Psi(Q) \coloneqq 2\sum_{q\le Q}\frac{\phi(q)\psi(q)}{q} .
	\label{eq:PsiDef}
	\end{equation}
 Aistleitner, Borda and Hauke \cite{ABH} proved that if $\lim_{Q\to\infty}\Psi(Q)=\infty$, then for almost all $\alpha$ and all $C>0$, we have 
	\begin{equation}
	\#\Bigl\{(a,q)\in\mathbb{Z}\times \N :\,q\le Q,\,\eqref{eq:Khintchine} \text{ holds},\,\gcd(a,q)=1\Bigr\}
		=\Psi(Q)+O_C\Bigl(\frac{\Psi(Q)}{(\log\Psi(Q))^C}\Bigr),
	\label{eq:QuantitativeDS}
	\end{equation}
	whenever $Q$ was sufficiently large in terms of $\alpha$ and $\psi$. Therefore the expected asymptotic holds for almost all $\alpha\in\R$, in keeping with \eqref{eq:QuantitativeKhintchine}. The main goal of our paper is to establish a refinement of \eqref{eq:QuantitativeDS}, which obtains a near-sharp error term similar to the one in \eqref{eq:QuantitativeKhintchine}.

	\begin{theorem} \label{thm:quantitative DS} Let $\psi:\N\to[0,1/2]$. Assume that $\Psi(Q)= 2\sum_{q\le Q} \psi(q)\phi(q)/q\to\infty$ as $Q\to\infty$.
		
		Then there exists a set $\CB$ of Lebesgue measure $0$ such that, for all $\alpha \in \R\setminus \CB$ and $\eps>0$,
		\[
		\#\bggg\{(a,q)\in\Z\times\N:\,q\le Q,\,\bgg|\alpha-\frac{a}{q}\bgg| < \frac{\psi(q)}{q},\,\gcd(a,q)=1\bggg\}
		= \Psi(Q) +  O_{\eps}\bgg(\Psi(Q)^{\frac{1}{2}+\eps}\bgg)
		\]
for all $Q$ sufficiently large in terms of $\alpha$ and $\psi$. 
	\end{theorem}
	LeVeque \cite{LeVeque2} showed that when $\psi$ satisfies certain more stringent conditions, the quantity on the left-hand side of \eqref{eq:QuantitativeDS} satisfies a Central Limit Theorem analogous to Fuchs' result \cite{Fuchs} for \eqref{eq:QuantitativeKhintchine}. In particular, the error term in Theorem \ref{thm:quantitative DS} cannot be $O(\Psi(Q)^{1/2})$, and so Theorem \ref{thm:quantitative DS} is sharp up to a factor of $\Psi(Q)^{o(1)}$. We have made no attempt to quantify more precisely this $\Psi(Q)^{o(1)}$ factor; as written the argument here would likely give a quantification $\Psi(Q)^{1/2+O((\log\log{\Psi(Q)})^{-c}) }$ for some $c>0$, but with a bit more care (particularly in Proposition \ref{prop:SmallPrimes}) this could surely be improved. 	
	
	\begin{rem*}
		LeVeque originally claimed a result analogous to Fuchs' theorem \cite{LeVeque1}. He corrected his claim in \cite{LeVeque2}, noticing that his proof is instead suited for the left-hand side of \eqref{eq:QuantitativeDS}. 
		
		 In Section \ref{sec:optimal}, we give a different and simple proof of the optimality of the $1/2$ in the exponent $1/2 + \epsilon$ in Theorem \ref{thm:quantitative DS}.
	\end{rem*}
	
%
%
%

\subsection*{Notation}

We shall use the letter $\lambda$ to denote the Lebesgue measure on $\R$. 

Sets will be typically denoted by capital calligraphic letters such as $\CA,\CV$ and $\CE$. A triple $G=(\CV,\CW,\CE)$ denotes a bipartite graph with vertex sets $\CV$ and $\CW$ and edge set $\CE\subseteq\CV\times\CW$. 

Given a set or an event $\CE$, we let $\mathbbm{1}_\CE$ denote its indicator function.

The letter $p$ will always denote a prime number. We also write $p^k\|n$ to mean that $p^k$ divides $n$, but $p^{k+1}$ does not divide $n$.

\section{Proof outline}

Throughout the rest of this paper, we fix a function $\psi:\N\to[0,1/2]$ such that $\Psi(Q)\rightarrow \infty$ (where $\Psi$ is as defined in \eqref{eq:PsiDef}) as $Q\rightarrow\infty$. In addition, we set
\[
N(\alpha;Q) =\#\Bigl\{(a,q)\in\Z\times \N :\,q\le Q,\ 
\bgg|\alpha-\frac{a}{q}\bgg|< \frac{\psi(q)}{q},\,\gcd(a,q)=1\Bigr\}.
\]
With this notation, the conclusion of Theorem \ref{thm:quantitative DS} is that
\[
N(\alpha;Q) =  \Psi(Q) +  O_{\eps}\bgg(\Psi(Q)^{\frac{1}{2}+\eps}\bgg) 
\]
for every fixed $\eps>0$, for almost all $\alpha\in\R$ and for all $Q$ sufficiently large in terms of $\alpha$. As a matter of fact, the function $\alpha\to N(\alpha;Q)$ is $1$-periodic, so that it suffices to prove the above estimate for almost all $\alpha\in[0,1]$.

Now, let
\begin{align*}
		\mathcal{A}_q&\coloneqq \bggg\{\alpha\in[0,1]: \bgg|\alpha-\frac{a}{q}\bgg|< \frac{\psi(q)}{q}\ \mbox{for some $a\in\Z$ coprime with $q$}\bggg\} \\ 
	&= [0,1]\cap \bigsqcup_{\substack{0 \leqslant a \leqslant q\\ \gcd(a,q)=1}} \left( \frac{a}{q} - \frac{\psi(q)}{q}, \frac{a}{q} + \frac{\psi(q)}{q} \right),
\end{align*}
where $\sqcup$ denotes a disjoint union. The fact that the above union is disjoint follows readily from our assumption that $\psi(q)\in[0,1/2]$. For the same reason, we restricted the union to $a\in[0,q]$, because otherwise the interval $(\frac{a-\psi(q)}{q},\frac{a+\psi(q)}{q})$ does not intersect $[0,1]$. We thus have
\begin{equation}
	\label{eq:N(a,Q)}
N(\alpha;Q) = \sum_{q\le Q} \un_{\CA_q}(\alpha) \quad\text{for all}\ \alpha\in[0,1],
\end{equation}
as well as
\[
\lambda (\mathcal{A}_q ) = \frac{2 \varphi(q) \psi(q)}{q}.
\]
As a consequence,
\begin{equation}
	\label{eq:N(a,Q) mean value}
\int_0^1 N(\alpha;Q)\dee\alpha = \sum_{q\le Q} \lambda(\CA_q) = \Psi(Q) .
\end{equation}

We shall deduce Theorem \ref{thm:quantitative DS} from the following bound on the variance of the random variable $[0,1]\ni\alpha\to N(\alpha;Q)$.

\begin{theorem}\label{thm:variance} Assume the above notation. Then, for every fixed $\eps>0$, we have
	\[
	\int_0^1 \bg(N(\alpha;Q)  - \Psi(Q) \bg)^2 \mathrm{d}\alpha
	\le \Psi(Q)+O_{\eps} \big(\Psi(Q)^{1+\eps}\big) \,;
	\]
	the implied constant depends at most on $\eps$.
\end{theorem}

\begin{rem*}
The inequality in Theorem \ref{thm:variance} holds  for all $\Psi(Q) \ge 0$. When $\Psi(Q)\ge 1$, the expression $ \Psi(Q)+O_{\eps} \big(\Psi(Q)^{1+\eps}\big)$ is dominated by $ \Psi(Q)^{1+\eps}$; when $\Psi(Q) < 1$, it  is dominated by $\Psi(Q)$.
\end{rem*}

As with all previous works, the key to Theorem \ref{thm:variance} is to show that $\lambda(\CA_q\cap\CA_r)\lessapprox \lambda(\CA_q)\lambda(\CA_r)$ for `most' pairs $q,r$, which can be thought of as a quantitative `quasi-independence on average' of the events $\CA_q$. Typically, one first attacks this via a counting argument, which shows
\[
\frac{\lambda(\CA_q\cap\CA_r)}{\lambda(\CA_q)\lambda(\CA_r)}\le \frac{\displaystyle\sum_{\gcd(c,n)=1}w\Bigl(\frac{c}{C}\Bigr)\sum_{d|\gcd(\ell,c)}f(d)}{\displaystyle\Bigl(\frac{\phi(n)}{n}\Bigr)\Bigl(\int_0^\infty w\Bigl(\frac{t}{C}\Bigr)dt\Bigr)\Bigl(\sum_{d|\ell}\frac{f(d)}{d}\Bigr)},
\]
 for some integers $n,\ell$ which divide $qr$, length of summation $C= \max(r\psi(q),q\psi(r))/\gcd(q,r)$, continuous compactly supported weight function $w$ and multiplicative function $f$. Thus, to show the right hand side is close to 1 we wish to show that the numerator, a weighted sum of a multiplicative function, is close to its expected size. The function $f$ and weight $w$ only cause minor technical inconveniences, but more substantial complications are caused by the restriction $\gcd(c,n)=1$. Thus morally the key issue is to get a good upper bound for the number of $c\le C$ which are coprime to $n$. We want an upper bound which is as close to the expected $\phi(n)C/n$ as possible, but with good uniformity in the size of $n$ and $C$. The work of Pollington and Vaughan \cite{PV} uses a standard upper bound sieve, which essentially gives
  \begin{equation}
 \sum_{\substack{c\le C\\ \gcd(c,n)=1}}1\ll \frac{\phi(n)}{n}C\prod_{\substack{p|n\\ p\ge C}}\Bigl(1+\frac{1}{p}\Bigr).
 \label{eq:PollingtonVaughan}
 \end{equation}
 The loss factor of the product over $p\ge C$ is typically close to 1, and the arguments of \cite{KM} use this to show $\lambda(\CA_q\cap\CA_r)\ll \lambda(\CA_q)\lambda(\CA_r)$ for `most' $q,r$, which is sufficient to resolve the Duffin-Schaeffer conjecture\footnote{An average form of $ \lambda(\CA_q\cap\CA_r) \ll  \lambda(\CA_q)\lambda(\CA_r)$ together with a simple application of  the Cauchy-Schwarz inequality would give $\lambda(\limsup_q \CA_q) > 0$. By Gallagher's 0-1 law \cite{gallagher}, one can immediately obtain $\lambda(\limsup_q \CA_q) = 1$, which confirms the Duffin-Schaeffer conjecture.}, but not to give better quantitative estimates.
 
 In \cite{ABH}, a fundamental-lemma style sieve was used, giving for any choice of the parameter $u$
   \begin{equation}
 \sum_{\substack{c\le C\\ \gcd(c,n)=1}}1\le \frac{\phi(n)}{n}C(1+O(u^{-u}))\prod_{\substack{p|n\\ p\ge C^{1/u} }}\Bigl(1+\frac{1}{p}\Bigr).
 \label{eq:FundamentalBound}
 \end{equation}
 This improves upon the Pollington-Vaughan bound because it no longer loses a constant factor if $u$ is moderately large and the final product is close to 1.  By choosing $u$ appropriately and using this bound in the arguments of \cite{KM}, they were able to show $\lambda(\CA_q\cap\CA_r)\le(1+o(1)) \lambda(\CA_q)\lambda(\CA_r)$ for most $q$, $r$, which then gave the result \eqref{eq:QuantitativeDS}. Unfortunately the quantification of the $o(1)$ is severely limited by the limitations on the possible choices of the parameter $u$ to ensure that the $O(u^{-u})$ error term is small but the product over primes $p>C^{1/u}$ is still typically close to 1, and it appears that there is no version of this argument which could yield a power-saving bound in \eqref{eq:QuantitativeDS}.
 
 In our work, we instead use a Legendre sieve to get an upper bound when $q$ and $r$ have suitably generic prime factorisations. Let $n^*$ denote the $t$-smooth part of $n$. Then we have
    \[
 \sum_{\substack{c\le C\\ \gcd(c,n)=1}}1\le  \sum_{\substack{c\le C\\ \gcd(c,n^*)=1}}1=C\frac{\phi(n^*)}{n^*}+O\Bigl(2^{\#\{p|n^*\}}\Bigr).
 \]
 When $t$ is chosen appropriately, this gives a good estimate provided $q$ and $r$ have somewhat typical factorisations. In particular, this gives a good quantitative estimate for suitably `generic' $q$, $r$. We then separately handle `non-generic' pairs $q$, $r$ using \eqref{eq:PollingtonVaughan} and adapting the methods of \cite{KM} to exploit the unusual factorisations. The crucial point is that for those `non-generic' pairs, we have an extra decay (roughly speaking, either exponential decay or polynomial decay) in the summation. Thus, the cruder Pollington--Vaughan inequality (cf.~Lemma \ref{lem:PV}) for $\lambda(\CA_q\cap\CA_r)$ is sufficient for our purpose.
 
 An adoption of the above strategy would enable one to get a result similar to Theorem \ref{thm:quantitative DS}, but with the exponent $1/2$ in the error term replaced by some number in $(1/2,1)$ close to $1$. To get the full quantitative strength of Theorem \ref{thm:quantitative DS} we need to modify some of the technical definitions in \cite{KM} (most notably, the definition of `quality' of a GCD graph) and verify that the arguments of \cite{KM} still hold with these quantitatively stronger definitions.

\begin{rem*}
Hauke--Saez--Walker \cite{HVW} have recently refined the argument of \cite{ABH} to obtain an estimate $\Psi(Q)+O(\Psi(Q)\exp(-(\log{\Psi(Q)})^{1/2+o(1)}))$ for the left hand side of \eqref{eq:QuantitativeDS}, which appears to be close to the strongest possible quantification obtainable from using only the overlap estimate \eqref{eq:FundamentalBound}. Their strategy (building on the earlier work of Green--Walker \cite{Green-Walker}) also contains some additional ideas similar to the modification of the definition of quality mentioned above.
\end{rem*}
	
\section{Structure of the paper}	

The first part of the paper consists of Sections \ref{Reduce:main:tovariance}-\ref{sec:three propositions} whose goal is to reduce the proof of Theorem \ref{thm:quantitative DS} to three bilinear estimates (Propositions \ref{prop:main-prop1}-\ref{prop:main-prop3}). Specifically, in Section \ref{Reduce:main:tovariance}, we show how Theorem \ref{thm:variance} implies Theorem  \ref{thm:quantitative DS}. In Section \ref{sec:Overlap} we establish a technical estimate about the measure of the intersection of two events $\CA_q$ and $\CA_r$; this is key for our improvement over the Aistleitner--Borda--Hauke estimate \eqref{eq:QuantitativeDS}. In Section \ref{sec:anatomy}, we present two auxiliary statistical lemmas about the multiplicative `anatomy' of a random integer; these serve as the underlying cause for the additional decay  in two of the bilinear estimates of Section \ref{sec:three propositions} (Propositions \ref{prop:main-prop2} and \ref{prop:main-prop3}). Lastly, in Section \ref{sec:three propositions}, we state the three key bilinear estimates (Propositions \ref{prop:main-prop1}-\ref{prop:main-prop3}) and deduce Theorem \ref{thm:variance} from them.

\medskip

The second part of the paper consists of Sections \ref{sec:GCDgraphs}-\ref{sec:IterationStep2}. In this part, we prove Propositions \ref{prop:main-prop1}-\ref{prop:main-prop3} using the language of GCD graphs developed in \cite{KM}. In Section \ref{sec:GCDgraphs},  we present the definition of a GCD graph and of other related notions, and we present three iterative Propositions (Propositions \ref{prop:IterationStep1}-\ref{prop:SmallPrimes})  and three structural lemmas (Lemmas \ref{lem:Cosmetic-L}-\ref{lem:HighDegreeSubgraph}) about GCD graphs. In Section \ref{sec:proof-of-moments-bound}, we show how the results of Section \ref{sec:GCDgraphs} imply the three bilinear estimates of Section \ref{sec:three propositions} (Propositions \ref{prop:main-prop1}-\ref{prop:main-prop3}). In Section \ref{sec:HighDegree}, we prove Lemma \ref{lem:HighDegreeSubgraph}. In Section \ref{sec:Prep}, we collect some preliminary properties of GCD graph.  Section \ref{sec:proof-anatomy-lemma} is  dedicated to the proof of the anatomical Lemmas \ref{lem:Cosmetic-L} and \ref{lem:Cosmetic-omega}, Section \ref{sec:IterationStep1} to the proof of Proposition \ref{prop:IterationStep1}. Finally, in Sections \ref{sec:SmallPrime} and \ref{sec:IterationStep2},  we prove Propositions \ref{prop:SmallPrimes} and \ref{prop:IterationStep2}, respectively.

\section{Proof of Theorem \ref{thm:quantitative DS} assuming Theorem \ref{thm:variance}}\label{Reduce:main:tovariance}

This claimed deduction follows by a straightforward adaption of Lemma 1.5 in \cite{Harman}. Let $Q_0=0$ and, for each $j\in\N$, let 
\[
Q_j \coloneqq \max\{Q\ge0 : \Psi(Q) < j \}.
\] 
We have that $\Psi(Q_j)\in[j-1,j)$ for all $j\ge1$. Hence, if $Q\in[Q_j,Q_{j+1})$, then $\Psi(Q)=j+O(1)$, and we also have that 
\[
N(\alpha;Q_j)\le N(\alpha;Q) \le N(\alpha;Q_{j+1}).
\]
This reduces Theorem \ref{thm:quantitative DS} to proving that
\begin{equation}
	\label{eq:quantitative DS}
 N(\alpha;Q_j) = \Psi(Q_j) +  O_\eps\bgg(\Psi(Q_j)^{\frac{1}{2}+\eps}\bgg)
\end{equation}
for almost all $\alpha\in[0,1]$ and all $j$ sufficiently large in terms of $\alpha$.

For each $r\in\N$, let $\CE_r$ be the set of $\alpha\in[0,1]$ such that
\[
\bg|  N(\alpha;Q_j) - \Psi(Q_j)\bg| > 2^{(\frac{1}{2}+\eps)r}  
\quad\text{for some}\ j\in[2^r,2^{r+1}),
\]
and let $\CB=\limsup_{r\rightarrow\infty} \CE_r$ be the set of $\alpha\in[0,1]$ which lie in infinitely many $\CE_r$.  Since we also have that $\Psi(Q_j)\asymp 2^r$ whenever $j\in[2^{r-1},2^r)$, we have \eqref{eq:quantitative DS} for all $\alpha\in[0,1]\setminus \CB$ and $j$ sufficiently large in terms of $\alpha$. Thus it suffices to show that $\lambda(\CB)=0$. We will prove that 
\begin{equation}
	\label{eq:BC for quantitative DS}
		\lambda(\CE_r) \ll_\eps 1/r^2 .
\end{equation}
This relation implies $\sum_{r=1}^\infty \lambda(\CE_r)<\infty$, so the first Borel--Cantelli lemma \cite[Lemma 1.2]{Harman} then shows that $\lambda(\CB)=0$, giving Theorem \ref{thm:quantitative DS}.

Now, let us fix  $r\in\N$ and demonstrate \eqref{eq:BC for quantitative DS}. Each integer $j\in[2^r,2^{r+1})$ can be written in binary form as $j=\sum_{t=0}^r b_t2^t$ with $b_t\in\{0,1\}$ for $t<r$ and $b_r=1$. For $s=0,1,\dots,r+1$, let us define $j_s=\sum_{t\ge s}b_t2^t$ with the convention that $j_{r+1}=0$. We then have that for any $j\in [2^r,2^{r+1})$
\begin{align*}
\sum_{q\le Q_j}   \bg(  \mathbbm{1}_{\mathcal{A}_q} (\alpha)  -  \lambda(\mathcal{A}_q) \bg) 
	&= \sum_{s=0}^r	\sum_{Q_{j_{s+1}}< q\le Q_{j_s} }  \bg(  \mathbbm{1}_{\mathcal{A}_q} (\alpha)  -  \lambda(\mathcal{A}_q) \bg)  \\
	&\le (r+1)^{1/2} \bgggg(\sum_{s=0}^r \bggg|\sum_{Q_{j_{s+1}}< q\le Q_{j_s} }  \bg(  \mathbbm{1}_{\mathcal{A}_q} (\alpha)  -  \lambda(\mathcal{A}_q) \bg)\bggg|^2 \bgggg)^{1/2}
\end{align*}
by the Cauchy--Schwarz inequality. Evidently, we may restrict the summation to those numbers $s$ for which $j_s>j_{s+1}$. For such $s$ we have that $j_s=j_{s+1}+2^s$ as well as $j_{s+1}=2^{s+1}i$ with $i<2^{r-s}$. Hence, if we set 
\[
\Delta(i, s, \alpha) \coloneqq 	
\sum_{Q_{i2^{s+1}}< q\le Q_{2^s+i2^{s+1}} }  \bg(  \mathbbm{1}_{\mathcal{A}_q} (\alpha)  -  \lambda(\mathcal{A}_q) \bg) ,
\]
then we conclude that
\[
\bggg|\sum_{q\le Q_j}   \bg(  \mathbbm{1}_{\mathcal{A}_q} (\alpha)  -  \lambda(\mathcal{A}_q) \bg) \bggg|^2
\le (r+1)  \sum_{0\le s\le r} \sum_{0\le i<2^{r-s}} \bg|\Delta(i,s,\alpha) \bg|^2  
\]
for each $j\in[2^r,2^{r+1})$. We now note that the right-hand side above doesn't depend on $j$. Hence, if $\alpha\in\CE_r$, then we must have
\[
\sum_{0\le s\le r} \sum_{0\le i<2^{r-s}} \bg|\Delta(i,s,\alpha) \bg|^2  > \frac{2^{(1+2\eps)r} }{r+1}.
\]
Using Markov's inequality, we arrive at the bound
\begin{equation}
\lambda(\CE_r) \le \frac{r+1}{2^{(1+2\eps)r}} \sum_{0\le s\le r} \sum_{0\le i<2^{r-s}} \int_0^1  \bg|\Delta(i,s,\alpha) \bg|^2\dee\alpha .
\label{eq:Markov}
\end{equation}
Fix for the moment $s$ and $i$ as above, and let $Q'=Q_{i2^{s+1}}$, $Q=Q_{2^s+i2^{s+1}}$ and $\psi'(q)=\un_{q>Q'}\psi(q)$. We see that
\[
\sum_{q\le Q}\frac{\psi'(q)\phi(q)}{q}=\sum_{Q'<q\le Q}\frac{\phi(q)\psi(q)}{q}=\Psi(Q)-\Psi(Q')=2^s+O(1),
\]
since $\Psi(Q_j)=j+O(1)$ for all $j\in\Z_{\ge0}$. To estimate the right-hand side of \eqref{eq:Markov}, we apply Theorem \ref{thm:variance} with the above choice of $Q$ and $\psi'(q)$ in place of $\psi$. As a consequence,
\[
\lambda(\CE_r) \ll_\eps  \frac{r+1}{2^{(1+2\eps)r}} \sum_{0\le s\le r} \sum_{0\le i<2^{r-s}} 2^{s(1+\eps)}  \ll_\eps \frac{r+1}{2^{\eps r}} \ll_\eps \frac{1}{r^2}.
\]
This proves \eqref{eq:BC for quantitative DS}, thus completing the deduction of Theorem \ref{thm:quantitative DS} from Theorem \ref{thm:variance}.

 \section{Overlap estimates}\label{sec:Overlap}
	
	Throughout the rest of the paper, we let
	\begin{equation}
D =	D(q,r)= \frac{\max \left(  r \psi(q), q \psi(r) \right)}{\gcd (q,r)}
	\label{eq:DDef}
	\end{equation}
	for any pair of natural numbers $(q,r)$. In addition, if $t\ge1$ is a real number, then we let
	\begin{equation}
	L_t(q,r)\coloneqq \sum_{\substack{p|qr/\gcd(q,r)^2 \\ p>t}}\frac{1}{p}
	\qquad\text{and}\qquad
	\omega_t(q,r)\coloneqq \#\bggg\{p\bgg|\frac{qr}{\gcd(q,r)^2}: p\le t \bggg\}.
	\label{eq:LtwtDef}
	\end{equation}
	
	It does not seem possible to improve the overlap estimate  in \cite[Lemma 5]{ABH} for arbitrary pair of integers $q$ and $r$, so we improve the overlap estimate for generic pairs of integers.

	\begin{lemma}[Overlap estimate for generic pairs of integers] \label{lem:ovelap estimate}
		Let $q\neq r$ be two natural numbers and let $D=D(q,r)$. For any real number $t\ge1$, we have 
		\begin{equation*} 
			\lambda (\mathcal{A}_q \cap \mathcal{A}_r) \leqslant 
			\un_{D\ge1/2}\cdot 
			\lambda (\mathcal{A}_q) \lambda (\mathcal{A}_r) e^{2L_t(q,r)} 
			\bggg( 1 + O \bggg(  \frac{2^{\omega_t(q,r)}\log(4D)}{D}   \bggg)\bggg).
		\end{equation*}
	\end{lemma}
	
	\begin{proof} If $D<1/2$, we know from \cite[Section 3]{PV} that $\mathcal{A}_q \cap \mathcal{A}_r=\emptyset$. So let us assume that $D\ge1/2$. Let $\nu_p(\cdot)$ denote the $p$-adic valuation function, let
		\[
		\ell= \prod_{p:\ \nu_p(q)=\nu_p(r)} p^{\nu_p(q)},
		\quad
		m=\prod_{p:\ \nu_p(q)\neq \nu_p(r)} p^{\min\{\nu_p(q),\nu_p(r)\}},
		\quad
		n=\prod_{p:\ \nu_p(q)\neq \nu_p(r)} p^{\max\{\nu_p(q),\nu_p(r)\}},
		\]
		and let
		\[
		w(y) = \begin{cases}
			2\delta&\text{if}\ 0\le y\le \Delta-\delta,\\
			\Delta+\delta-y &\text{if}\ \Delta-\delta<y\le \Delta+\delta,\\
			0 &\text{otherwise},
		\end{cases}
		\]
		with
		\[
		\delta=\min\bggg\{\frac{\psi(q)}{q},\frac{\psi(r)}{r}\bggg\}
		\quad\text{and}\quad
		\Delta=\max\bggg\{\frac{\psi(q)}{q},\frac{\psi(r)}{r}\bggg\} .
		\]
		Finally, let $s=\un_{2|\ell}$. Then, arguing as in the proof of Lemma 5 of \cite{ABH}, we have 
		\[
		\lambda (\mathcal{A}_q \cap \mathcal{A}_r)
		= 2^{s+1} \phi(m)\frac{\phi(\ell)^2}{\ell} \prod_{\substack{p|\ell \\ p>2}}\bigg(1-\frac{1}{(p-1)^2}\bigg)
		\sum_{\substack{c\ge 1\\ \gcd(c,n)=1}} w\bgg(\frac{2^sc}{\ell n}\bgg) \prod_{\substack{p|\gcd(\ell,c) \\ p>2}} \bggg(1+\frac{1}{p-2}\bggg) .
		\]
		Since $w$ is supported on $[0,\Delta+\delta]\subseteq[0,2\Delta]$, we must have $c\le 2\Delta\ell n=2D$. In particular, we may restrict the last product to primes $p\le 2D$. 
		In addition, we may replace the condition $\gcd(c,n)=1$ by the condition $\gcd(c,n^*)=1$, where $n^*$ denotes the $t$-smooth part of $n$, at the cost of replacing the exact expression for $\lambda(\CA_q\cap\CA_r)$ by an upper bound. We detect the condition that $\gcd(c,n^*)=1$ using M\"obius inversion. 
		
		In conclusion, the above discussion implies that
	\[
	\lambda (\mathcal{A}_q \cap \mathcal{A}_r)
	\le  2^{s+1} \phi(m)\frac{\phi(\ell)^2}{\ell} \prod_{\substack{p|\ell \\ p>2}}\bigg(1-\frac{1}{(p-1)^2}\bigg)
	\sum_{a|n^*}\mu(a) 
	\sum_{\substack{c\ge 1 \\ a|c}} w\bgg(\frac{2^sc}{\ell n}\bgg) \prod_{\substack{p|\gcd(\ell,c) \\ 2<p\le 2D}} \bggg(1+\frac{1}{p-2}\bggg) 
	\]
	with $\mu$ denoting the M\"obius function throughout this proof. If we let $c=ab$, then $\gcd(\ell,c)=\gcd(\ell,b)$ because $a|n^*$, and $n^*$ and $\ell$ are co-prime. Moreover, note that
	\[
	\prod_{\substack{p|k \\ 2<p\le 2D}}\bggg(1+\frac{1}{p-2}\bggg) = \sum_{j|k}f(j),
	\]
	where $f$ is the unique multiplicative function that is supported on square-free integers and satisfies $f(p)=\un_{2<p\le 2D}/(p-2)$. 
	Hence,
	\begin{align*}
		\lambda (\mathcal{A}_q \cap \mathcal{A}_r)
		&\le  2^{s+1} \phi(m)\frac{\phi(\ell)^2}{\ell} \prod_{\substack{p|\ell \\ p>2}}\bigg(1-\frac{1}{(p-1)^2}\bigg)
		\sum_{a|n^*}\mu(a) \sum_{j|\ell} f(j) 
		\sum_{i\ge 1} w\bgg(i\cdot \frac{2^saj}{\ell n}\bgg)  ,
	\end{align*}
	where we set $b=ij$. Uniformly for $\varrho>0$, the Euler--Maclaurin formula\footnote{In \cite{dk-book}, the Euler--Maclaurin formula is stated for continuously differentiable functions, but it can be easily generalized to continuous functions that are piecewise differentiable. This requires Theorem 7.35 in \cite{Apostol}.} \cite[Theorem 1.10]{dk-book} implies 
	\begin{equation}
		\label{eq:w partial summation}
	\sum_{i\ge 1} w(\varrho i) = \int_0^\infty w(\varrho t)\dee t +\int_0^\infty \varrho w'(\varrho t)\{t\}\dee t
	= 2\delta\Delta \varrho^{-1} + O(\delta). 
		\end{equation}
	Therefore,
		\begin{align*}
		\lambda (\mathcal{A}_q \cap \mathcal{A}_r)
		&\le  4\delta\Delta \ell n \phi(m)\frac{\phi(\ell)^2}{\ell} \prod_{\substack{p|\ell \\ p>2}}\bigg(1-\frac{1}{(p-1)^2}\bigg)
		\sum_{a|n^*} \frac{\mu(a)}{a} \sum_{j|\ell} \frac{f(j)}{j}  \\
		&\quad + O\bggg( \delta \phi(m)\frac{\phi(\ell)^2}{\ell} \prod_{\substack{p|\ell \\ p>2}}\bigg(1-\frac{1}{(p-1)^2}\bigg)
		\sum_{a|n^*} \mu^2(a) \sum_{j|\ell} f(j) \bggg).
	\end{align*}
	We have 
	\[
	\sum_{a|n^*} \frac{\mu(a)}{a} = \frac{\phi(n^*)}{n^*},\quad 
	\sum_{a|n^*} \mu^2(a) = 2^{\#\{p|n^*\}}  = 2^{\omega_t(q,r)},
	\]
	as well as
	\[
	\sum_{j|\ell}\frac{f(j)}{j} \le \prod_{\substack{p|\ell \\ p>2}} \bggg(1-\frac{1}{(p-1)^2}\bggg)^{-1} ,\quad
	\sum_{j|\ell} f(j) =  \prod_{\substack{p|\ell \\ 2<p\le 2D}} \bggg(1+\frac{1}{p-2}\bggg)\ll \log(4D) .
	\]
	Since we also have that $D=\Delta \ell n$ and $4\delta\Delta \phi(\ell)^2\phi(m)\phi(n)=\lambda(\CA_q)\lambda(\CA_r)$, we conclude that
	\[
		\lambda (\mathcal{A}_q \cap \mathcal{A}_r)
		\le  \lambda(\CA_q)\lambda(\CA_r)  \cdot \frac{\phi(n^*)/n^*}{\phi(n)/n} \cdot 
		\bggg(1+ O\bggg( \frac{2^{\omega_t(q,r)} \log(4D)}{D}   \bggg).
	\]
	In order to complete the proof, note that
	\[
	 \frac{\phi(n^*)/n^*}{\phi(n)/n} = \prod_{p|n,\ p>t}\bggg(1+\frac{1}{p-1}\bggg)
	 \le e^{2L_t(q,r)},
	\]
	where we used the inequality $1+1/(y-1)\le e^{2/y}$ for $y\ge2$. 
	\end{proof}

	For non-generic pairs of integers, we use the following estimate.
	
	\begin{lemma}[Pollington--Vaughan \cite{PV}]\label{lem:PV}
	If $q,r$ and $D$ are as in Lemma \ref{lem:ovelap estimate}, then we have
		\begin{equation*} 
			\lambda(\mathcal{A}_q \cap \mathcal{A}_r) \ll \un_{D\ge1/2}\cdot \lambda(\mathcal{A}_q) \lambda(\mathcal{A}_r) e^{L_D(q,r)}  .
		\end{equation*} 
	\end{lemma}

\section{Auxiliary anatomical lemmas}\label{sec:anatomy}

\begin{lemma}[Bounds on multiplicative functions]\label{lem:Tenenbaum}
	Let $k\in\N$ and write $\tau_k$ for the $k$-th divisor function. If $f$ is a multiplicative function such that $0\le f\le \tau_k$, then
	\[
	\sum_{n\le x} f(n) \ll_k x \cdot \exp\Big\{\sum_{p\le x}\frac{f(p)-1}{p}\Big\} .
	\]
\end{lemma}

\begin{proof}
	This is \cite[Theorem 14.2, p. 145]{dk-book}.
\end{proof}

\begin{lemma}[Few numbers with many prime factors]\label{lem:Anatomy-L}
	Fix $C\ge1$. For $x,t,s\ge1$, we have the uniform estimate
	\[
	\#\bigg\{n\le x: \sum_{\substack{p|n\\  p>t}}\frac{1}{p} \ge  \frac{1}{s} \bigg\} \ll_C xe^{-Ct/s} .
	\]
\end{lemma}

\begin{proof}
The quantity in question is
\[
\le e^{-Ct/s} \sum_{n\le x}f(n)
\]
with $f(n)=\prod_{p|n} e^{\un_{p>t} Ct/p}$. In particular, $f(p)=1+O_C(\un_{p>t}t/p)$. Hence we may use Lemma \ref{lem:Tenenbaum} to complete the proof. 
\end{proof}

\begin{lemma}[Few numbers with many prime factors]\label{lem:Anatomy-omega}
	Fix $C\ge1$ and $\kappa>0$. For $x, t \ge1$, we have the uniform estimate
	\[
	\#\bgg\{n\le x: \#\{p|n:p\le t\} \ge \kappa\log t \bgg\} \ll_{C,\, \kappa} x \cdot t^{-C} .
	\]
\end{lemma}

\begin{proof}
	The quantity in question is
	\[
	\le t^{-C-1} \sum_{n\le x}f(n)
	\]
	with $f(n)=\prod_{p|n} e^{ \un_{p\le t} \frac{C+1}{\kappa}}$. Using Lemma \ref{lem:Tenenbaum} completes the proof. 
\end{proof}

	\section{Reduction to three propositions}\label{sec:three propositions}

	In this section, we  reduce the proof of Theorem  \ref{thm:variance} (and hence Theorem \ref{thm:quantitative DS}) to the following three propositions. In their statements, recall the definitions of $D(q,r)$, $L_t(q,r)$ and $\omega_t(q,r)$ given in the beginning of Section \ref{sec:Overlap}.

	\begin{proposition}
		\label{prop:main-prop1}
		Fix $\eps \in (0, 1)$. Let $\psi:\N\to\R_{\ge0}$ and $y\ge1$. We then have the uniform estimate
		\[
		\mathop{\sum\sum}_{\substack{(q,r)\in[1,Q]^2 \\ D(q,r)\le y}} \frac{\psi(q)\phi(q)}{q}\cdot \frac{\psi(r)\phi(r)}{r}  \ll_\eps y^{1-\eps} \Psi(Q)^{1+\eps} .
		\]
	\end{proposition}
	
		\begin{proposition}
		\label{prop:main-prop2}
		Fix $\eps\in (0, 1)$ and $C\ge1$. Let $\psi:\N\to\R_{\ge0}$ and $y,t,s\ge1$. We then have the uniform estimate
		\[
		\mathop{\sum\sum}_{\substack{(q,r)\in[1,Q]^2 \\ D(q,r)\le y,\ L_t(q,r)\ge 1/s}} \frac{\psi(q)\phi(q)}{q}\cdot \frac{\psi(r)\phi(r)}{r}  \ll_{\eps, \,C} 
			e^{-Ct/s} y^{1-\eps} \Psi(Q)^{1+\eps} .
		\]
	\end{proposition}

	\begin{proposition}
	\label{prop:main-prop3}
	Fix $\eps \in (0, 1),\kappa>0$ and $C\ge1$. Let $\psi:\N\to\R_{\ge0}$ and $y,t\ge1$. We then have the uniform estimate
	\[
	\mathop{\sum\sum}_{\substack{(q,r)\in[1,Q]^2 \\ D(q,r)\le y,\ \omega_t(q,r)\ge \kappa\log t}} \frac{\psi(q)\phi(q)}{q}\cdot \frac{\psi(r)\phi(r)}{r}  \ll_{\eps,\,\kappa,\,C} 
	t^{-C} y^{1-\eps} \Psi(Q)^{1+\eps} .
	\]
\end{proposition}

	\begin{proof}[Deduction of Theorem \ref{thm:variance} from Propositions \ref{prop:main-prop1}-\ref{prop:main-prop3}]
		Using \eqref{eq:N(a,Q)} and \eqref{eq:N(a,Q) mean value}, we find that
		\begin{align*}
			0\le \int_0^1 \bg(N(\alpha;Q)-\Psi(Q)\bg)^2\dee\alpha	
			&= \int_0^1\bgg(\sum_{q\le Q}\un_{\CA_q}\bgg)^2\dee\alpha- \Psi(Q)^2 \\
			&=  \mathop{\sum\sum}_{q,r \leqslant Q} \lambda( \mathcal{A}_q \cap \mathcal{A}_r) - \Psi(Q)^2 .
		\end{align*}
		So in order to prove Theorem \ref{thm:variance},  it suffices to establish the upper bound 
		\[
			\mathop{\sum\sum}_{ q,r \leqslant Q} \lambda( \mathcal{A}_q \cap \mathcal{A}_r) \leqslant  \Psi(Q)^2 + \Psi(Q)+ O_{\eps} \bg( \Psi(Q)^{1+\eps}\bg).
		\]
		The terms with $q=r$ contribute 
		\[
		\sum_{q\le Q}\lambda(\CA_q)=\Psi(Q). 
		\]
		It thus remains to show that 
		\begin{align}\label{eq:goal}
				\mathop{\sum\sum}_{ \substack{q,r \leqslant Q \\  q\neq r}} 
					\lambda( \mathcal{A}_q \cap \mathcal{A}_r) \leqslant  \Psi(Q)^2 +  O_{\eps} \bg( \Psi(Q)^{1+\eps}\bg).
		\end{align}
		
		From the first part of Lemma \ref{lem:ovelap estimate}, we may assume that $D(q,r)\ge1/2$. To this end, let 
		\[
		\CE = \{(q,r)\in[1,Q]^2: q\neq r,\ D(q,r)\ge1/2\},
		\]
		so that the sum in \eqref{eq:goal} is over all pairs $(q,r)\in\CE$. Writing $D=D(q,r)$, we split $\CE$ into the following three subsets:
		\begin{align*}
			\CE^{(1)} &= \bgg\{ (q,r)\in\CE : L_{D^2}(q,r)\le \frac{1}{D},\ \omega_{D^2}(q,r) \le \frac{\eps}{4} \log(2D) \bgg\},\\
			\CE^{(2)} &= \bgg\{ (q,r) \in \CE : L_{D^2}(q,r)> \frac{1}{D} \bgg\},\\
			\CE^{(3)} & = \bgg\{ (q,r)\in \CE:  L_{D^2}(q,r)\le  \frac{1}{D},
			\ \omega_{D^2}(q,r) > \frac{\eps}{4} \log(2D)  \bgg\} .
		\end{align*}
	We claim the following estimates:
	\begin{align}
	\mathop{\sum\sum}_{(q,r)\in\CE^{(1)}} \lambda( \mathcal{A}_q \cap \mathcal{A}_r)
		&\le \Psi(Q)^2+O_\eps\bg(\Psi(Q)^{1+\eps}\bg) \,;
						\label{eq:E1 contribution}  \\
	\mathop{\sum\sum}_{(q,r)\in\CE^{(2)}} \lambda( \mathcal{A}_q \cap \mathcal{A}_r) 
		&\ll_\eps\Psi(Q)^{1+\eps} \,; 
					\label{eq:E2 contribution} \\
	\mathop{\sum\sum}_{(q,r)\in\CE^{(3)}} \lambda( \mathcal{A}_q \cap \mathcal{A}_r)
		&\ll_\eps\Psi(Q)^{1+\eps}.
	\label{eq:E3 contribution} 
	\end{align}	
	These bounds suffice to complete the proof, since they readily imply
	\[
				\mathop{\sum\sum}_{ \substack{q,r \leqslant Q \\  q\neq r}} 
						\lambda( \mathcal{A}_q \cap \mathcal{A}_r)=\sum_{i=1}^3\mathop{\sum\sum}_{ \substack{(q,r) \in \CE^{(i)}}} 
						\lambda( \mathcal{A}_q \cap \mathcal{A}_r)\le \Psi(Q)^2+O_\eps\Bigl(\Psi(Q)^{1+\eps}\Bigr).
	\]
	This gives \eqref{eq:goal}, and so completes the proof of Theorem \ref{thm:variance}. 
	
	Let us now prove \eqref{eq:E1 contribution}. Here the summation is over pairs $(q,r)\in\CE^{(1)}$. We then use Lemma \ref{lem:ovelap estimate} with $t=D^2$ to find that
	\[
		\lambda (\mathcal{A}_q \cap \mathcal{A}_r) \leqslant
		 \lambda (\mathcal{A}_q) \lambda (\mathcal{A}_r) \bg( 1 + O\bg(D^{-1+\eps/2}\bg)\bg) .
	\]
	Therefore
	\[
			\mathop{\sum\sum}_{(q,r)\in\CE^{(1)}} \lambda( \mathcal{A}_q \cap \mathcal{A}_r)
		\le M+O_\eps(R),
	\]
	where 
	\[
	M\coloneqq \mathop{\sum\sum}_{(q,r)\in\CE^{(1)}} \lambda( \mathcal{A}_q) \lambda(\mathcal{A}_r) 
	\quad\text{and}\quad
	R\coloneqq \mathop{\sum\sum}_{(q,r)\in\CE^{(1)}} D(q,r)^{\eps/2-1} \lambda( \mathcal{A}_q) \lambda(\mathcal{A}_r).
	\] 
	For the main term $M$ we simply notice that it is $\le \Psi(Q)^2$. Thus we focus on the $O_\eps(R)$ term. Let $I$ be the smallest non-negative integer such that $2^I\ge\Psi(Q)$. We then split $R$ according to the size of $D(q,r)$ as
	\[
	R = \sum_{i=0}^{I} R_i+R'
	\]
	where $R_i$ is the part of the sum with $D(q,r)\in[2^{i-1},2^i)$, and $R'$ is the part with $D(q,r)\ge 2^I\ge  \Psi(Q)$. Now, we trivially have that
	\[
	R'\le \Psi(Q)^{\eps/2-1} \mathop{\sum\sum}_{(q,r)\in\CE^{(1)}} \lambda( \mathcal{A}_q) \lambda(\mathcal{A}_r)\le \Psi(Q)^{1+\eps/2}.
	\]
	For each $i\in\{0,1,\dots,I\}$, Proposition \ref{prop:main-prop1} implies that
	\[
	R_i\ll_\eps (2^i)^{\eps/2-1} \cdot (2^i)^{1-\eps} \Psi(Q)^{1+\eps} = 2^{-i\eps/2} \Psi(Q)^{1+\eps}.
	\]
	We thus conclude that $R\ll_\eps \Psi(Q)^{1+\eps}$, whence \eqref{eq:E1 contribution} follows. 
	
	Let us now prove \eqref{eq:E2 contribution}. For $(q,r)\in\CE^{(2)}$ we use Lemma \ref{lem:PV} to find that 
	\[
	\mathop{\sum\sum}_{(q,r)\in\CE^{(2)}} \lambda( \mathcal{A}_q \cap \mathcal{A}_r)
	\ll 		\mathop{\sum\sum}_{(q,r)\in\CE^{(2)}} \lambda( \mathcal{A}_q ) \lambda( \mathcal{A}_r)
	e^{L_D(q,r)} .
	\]
	Note that if $D\in[2^{i-1},2^i)$ and $(q,r)\in\CE^{(2)}$, then $L_{4^{i-1}}(q,r)>1/2^i$. We may thus write
	\[
	\CE^{(2)} = \bigcup_{i=0}^\infty \bigcup_{j=i}^\infty \CE^{(2)}_{i,j},
	\]
	where $\CE^{(2)}_{i,i}$ is the set of pairs $(q,r)\in \CE^{(2)}$ with $D\in[2^{i-1},2^i)$ and $L_{4^{i-1}}(q,r)\le 1$, and $\CE^{(2)}_{i,j}$ with $j>i$ is  the set of pairs $(q,r)\in \CE^{(2)}$ with $D\in[2^{i-1},2^i)$ and  $L_{4^{j-1}}(q,r)\le 1<L_{4^{j-2}}(q,r)$. In particular, if $(q,r)\in\CE^{(2)}_{i,j}$, then 
	\[
	L_D(q,r)\le L_{2^{i-1}}(q,r) \le \sum_{2^{i-1}<p\le 4^{j-1}}\frac{1}{p} + L_{4^{j-1}}(q,r)\le \log\frac{j+1}{i+1} + O(1),
	\]
	and thus
	\begin{equation}
		\label{eq:E2 contribution - almost there}
	\mathop{\sum\sum}_{(q,r)\in\CE^{(2)}} \lambda( \mathcal{A}_q \cap \mathcal{A}_r)
	\ll 	\mathop{\sum\sum}_{j\ge i\ge 0}\frac{j+1}{i+1} 
		\mathop{\sum\sum}_{(q,r)\in\CE_{i,j}^{(2)}} \lambda( \mathcal{A}_q ) \lambda( \mathcal{A}_r).
	\end{equation}
For each $i\ge0$, Proposition \ref{prop:main-prop2} applied with $y=2^{i}$, $t=\max\{1,4^{i-1}\}$, $s=2^i$ and $C=8$ implies that\footnote{Since any prime is larger than $1$, we have $L_{4^{i-1}}(q,r)=L_1(q,r)$ when $i<1$, so we may indeed take $t=\max\{1,4^{i-1}\}$.}
	\[
		\mathop{\sum\sum}_{(q,r)\in\CE_{i,i}^{(2)}} \lambda( \mathcal{A}_q ) \lambda( \mathcal{A}_r)
		\ll_\eps e^{-2^i} \Psi(Q)^{1+\eps}.
	\]
 In addition, if $j>i\ge0 $, then Proposition \ref{prop:main-prop2} applied with $y=2^i$, $t=\max\{1,4^{j-2}\}$, $s=1$ and $C=17$ implies that
	\[
	\mathop{\sum\sum}_{(q,r)\in\CE_{i,j}^{(2)}}   \lambda( \mathcal{A}_q ) \lambda( \mathcal{A}_r)
	\ll_\eps e^{-4^j} \Psi(Q)^{1+\eps}.
	\]
Inserting the two above bounds into \eqref{eq:E2 contribution - almost there} completes the proof of \eqref{eq:E2 contribution}.

It remains to prove \eqref{eq:E3 contribution}. We note that for pairs $(q,r)\in\CE^{(3)}$ we have 
\[
L_D(q,r)\le \sum_{D< p\le D^2}\frac{1}{p}+L_{D^2}(q,r)\ll 1.
\]
Therefore, using Lemma \ref{lem:PV}, we find that
\[
\mathop{\sum\sum}_{(q,r)\in\CE^{(3)}} \lambda( \mathcal{A}_q \cap \mathcal{A}_r)
\ll 		\mathop{\sum\sum}_{(q,r)\in\CE^{(3)}} \lambda( \mathcal{A}_q ) \lambda( \mathcal{A}_r) .
\]
Note that if $D\in[2^{i-1},2^i)$ and $\omega_{D^2}(q,r)>0.25\eps\log(2D)$, then $\omega_{4^i}(q,r)>0.1\eps\log(4^i)$. Therefore,
\begin{equation}
\mathop{\sum\sum}_{(q,r)\in\CE^{(3)}} \lambda( \mathcal{A}_q \cap \mathcal{A}_r)
\ll \sum_{i\ge0} 
\mathop{\sum\sum}_{(q,r)\in\CE_i^{(3)}} \lambda( \mathcal{A}_q ) \lambda( \mathcal{A}_r),
	\label{eq:E3 contribution - almost there}
\end{equation}
where $\CE^{(3)}_i$ is the set of pairs $(q,r)\in \CE^{(3)}$ with $D\in[2^{i-1},2^{i})$ and $\omega_{4^i}(q,r)>0.1\eps\log(4^i)$. Proposition \ref{prop:main-prop3} applied with $y=2^i$, $t=4^i$, $\kappa=\eps/10$, and $C=2$ implies that
\[
\mathop{\sum\sum}_{(q,r)\in\CE_i^{(3)}}   \lambda( \mathcal{A}_q ) \lambda( \mathcal{A}_r)
\ll_\eps 2^{-i} \Psi(Q)^{1+\eps}
\]
for all $i\ge0$. Inserting this bound into \eqref{eq:E3 contribution - almost there} completes the proof of \eqref{eq:E3 contribution}, and thus of Theorem \ref{thm:variance}.
\end{proof}

	\section{GCD graphs}\label{sec:GCDgraphs}

	In order to prove Propositions \ref{prop:main-prop1}-\ref{prop:main-prop3}, we shall use the method of GCD graphs developed in \cite{KM} with some important modifications.  The definition of a  GCD graph is almost the same as in \cite{KM}:

	\begin{dfn}[GCD graph]\label{gcd graph dfn}
		Let $G$ be a septuple $(\mu,\CV,\CW,\CE,\CP,f,g)$ such that:	
		\begin{enumerate}
			\item $\mu$ is a measure on $\N$ such that $\mu(n)<\infty$ for all $n\in\N$; we extend to $\N^2$ by letting
			\[
			\mu(\CN):=\sum_{(n_1,n_2)\in\CN}\mu(n_1)\mu(n_2) 
			\quad\text{for}\quad\CN\subseteq\N^2;	
			\]	
			\item $\CV$ and $\CW$ are finite sets of positive integers;
			\item $\CE\subseteq\CV\times \CW$, that is to say $(\CV,\CW,\CE)$ is a bipartite graph;
			\item $\CP$ is a set of primes;
			\item $f$ and $g$ are functions from $\CP$ to $\Z_{\ge0}$ such that for all $p\in\CP$ we have:
			\begin{enumerate}
				\item $p^{f(p)}|v$ for all $v\in\CV$, and $p^{g(p)}|w$ for all $w\in\CW$;
				\item if $(v,w)\in\CE$, then $p^{\min\{f(p),g(p)\}}\|\gcd(v,w)$;
			\end{enumerate}
		\end{enumerate}			
		We then call $G$ a (bipartite) \emph{GCD graph} with \emph{sets of vertices} $(\CV,\CW)$, \emph{set of edges} $\CE$ and \emph{multiplicative data} $(\CP,f,g)$. We will also refer to $\CP$ as the \emph{set of primes} of $G$. If $\CP=\emptyset$, we say that $G$ has \emph{trivial} set of primes and we view $f=f_\emptyset$ and $g=g_\emptyset$ as two copies of the empty function from $\emptyset$ to $\mathbb{Z}_{\ge 0}$.
	\end{dfn}
	
	\begin{rem*}
	In \cite{KM}, the definition of a GCD graph included one additional condition: if $f(p)\neq g(p)$, then we required that $p^{f(p)}\|v$ for all $v\in\CV$, and $p^{g(p)}\|w$ for all $w\in\CW$. Even though the GCD graphs we will consider do satisfy this condition, we no longer need to include it in the definition due to a small strengthening of Proposition \ref{prop:IterationStep1}, which allows us to specify the $p$-adic valuation of all vertices in $G'$ and for all primes $p\in\CP'\setminus \CP$. 
	\end{rem*}

 To each GCD graph $G$, we associate a {\it quality} $q(G)$. This will be defined differently to \cite{KM}, and it will depend on some parameter $\tau\in(0,\frac{1}{100})$, which we can take to be arbitrarily small.
	
	\begin{dfn}
		For a GCD graph $G = (\mu,\mathcal{V},\mathcal{W},\mathcal{E},\mathcal{P},f,g)$, we define: 
		\begin{enumerate}[label=(\alph*)]
			\item The \emph{edge density} is defined by
			$$
			\delta(G) = \frac{\mu(\mathcal{E})}{\mu(\mathcal{V}) \mu(\mathcal{W})},
			$$
			provided that $\mu(\mathcal{V}) \mu(\mathcal{W}) \neq 0$.
			If $\mu(\mathcal{V}) \mu(\mathcal{W}) = 0$, we define $\delta(G)$ to be $0$.
			
			\item The \emph{quality} is defined by
			\[
			q(G):=\delta^{2+\tau}\mu(\mathcal{V})\mu(\mathcal{W})\prod_{p\in \mathcal{P}}
			p^{|f(p)-g(p)|}
			\bggg(1-\frac{\mathbbm{1}_{f(p)=g(p)\geqslant 1}}{p}\bggg)^{-2}  
			\bggg(1- \frac{1}{p^{1+\frac{\tau}{4}}} \bggg)^{-3}  .
			\]
		\end{enumerate}		
\end{dfn}

\begin{dfn}[Non-trivial GCD graph]\label{non-trivial gcd graph dfn}
Let $G=(\mu,\CV,\CW,\CE,\CP,f,g)$. We say that $G$ is {\it non-trivial} if $\mu(\CE)>0$. 
\end{dfn}

\begin{dfn}[GCD subgraph]
	Let $G=(\mu,\CV,\CW,\CE,\CP,f,g)$ and $G'=(\mu',\CV',\CW',\CE',\CP',f',g')$ be two GCD graphs. We say that $G'$ is a \emph{GCD subgraph} of $G$ if:
	\[
	\mu'=\mu,\quad \CV'\subseteq\CV,\quad \CW'\subseteq\CW,\quad \CE'\subseteq\CE,
	\quad \CP'\supseteq\CP, \quad f'\big|_{\CP}=f,\quad
	g'\big|_{\CP}=g.
	\]
	 Lastly, we say that $G'$ is a {\it non-trivial GCD subgraph} of $G$ if $\mu(\CE')>0$, that is to say $G'$ is non-trivial as a GCD graph.
\end{dfn}

\begin{dfn}[maximal GCD graph]\label{maximal GCD graph}
	Let $G = (\mu,\mathcal{V},\mathcal{W},\mathcal{E},\mathcal{P},f,g)$ be a GCD graph. We say that $G$ is \emph{maximal} or, more precisely, that $G$ is $(\CP,f,g)$-{\it maximal} if for every GCD subgraph $G'$ of $G$ with the same multiplicative data $(\CP,f,g)$ we have that $q(G')\le q(G)$.
\end{dfn}

\begin{rem*}
For any given GCD graph $G= (\mu,\mathcal{V},\mathcal{W},\mathcal{E},\mathcal{P},f,g)$, since both the sets $\mathcal{V}$ and $\mathcal{W}$ are finite, there exists a  GCD subgraph $G'$ of $G$ with the same multiplicative data  $(\CP,f,g)$, such that $G'$ is $(\CP,f,g)$-maximal. To see this, just list all  subgraphs of  $G$ with the same multiplicative data, and pick out one with the largest possible quality.  We will use this fact several times in the later sections without mentioning it again.
	\end{rem*}

We recall from \cite{KM} the special vertex sets $\CV_{p^k}$ and GCD graphs $G_{p^k,p^\ell}$ formed by restricted elements by their prime-power divisibilty.

\begin{dfn}
\label{def:special graphs}	
Let $p$ be a prime number, and let $k,\ell\in\Z_{\ge0}$.
\begin{enumerate}
\item If $\CV$ is a set of integers, we set
	\[
	\CV_{p^k}=\{v\in\CV:p^k\|v\}.
	\]
Here we have the understanding that $\CV_{p^0}$ denotes the set of $v\in\CV$ that are coprime to $p$.

\item Let $G=(\CV,\CW,\CE)$ be a bipartite graph. If $\CV'\subseteq\CV$ and $\CW'\subseteq\CW$, we define
\[
\CE(\CV',\CW'):=\CE\cap(\CV'\times\CW').
\]
We write for brevity
\[
\CE_{p^k,p^\ell}:=\CE(\CV_{p^k},\CW_{p^\ell}) .
\]

\item Let $G=(\mu,\CV,\CW,\CE,\CP,f,g)$ be a GCD graph such that $p\notin \CP$. We then define the septuple
\[
G_{p^k,p^\ell}=(\mu,\CV_{p^k},\CW_{p^\ell},\CE\cap(\CV_{p^k}\times \CW_{p^\ell}),\CP\cup\{p\},f_{p^k},g_{p^\ell})
\]
where the functions $f_{p^k}$, $g_{p^\ell}$ are defined on $\CP\cup\{p\}$ by the relations $f_{p^k}\vert_\CP=f$, $g_{p^\ell}\vert_\CP=g$, 
\[
f_{p^k}(p)=k\quad\text{and}\quad g_{p^\ell}(p)=\ell.
\]
\end{enumerate}
\end{dfn}

We further fix a constant $M\ge2$ and we introduce the auxiliary quantities  $C_1,\dots,C_7$ as follows:
\begin{equation}\label{eq:CDefs}
	\begin{split}
	&C_1
	=\frac{10^4}{\tau},  \qquad 
	C_2 = 10MC_1^3,  
	\qquad  
	C_3 = 10^3 C_1^3 , 
	\qquad 
	C_4 = 10^{10}M^2 C_2^2, \\
	&C_5=\max\{C_3, (50\log C_4)^3\}, 
	\qquad  C_6 = \max\{C_4, 10^4 MC_2, C_2^{ \frac{10}{\tau}}\},  
	\qquad C_7=C_5^{C_6}.
	\end{split}
\end{equation}

	\begin{dfn} Let $G = (\mu,\mathcal{V},\mathcal{W},\mathcal{E},\mathcal{P},f,g)$ be a GCD graph.
	We let $\CR(G)$ be given by
			\[
			\CR(G):= \{p\notin\CP: \exists (v,w)\in\CE\ \text{such that}\ p|\gcd(v,w)\} .
			\]
			That is to say $\CR(G)$ is the set of primes occurring in a GCD which we haven't yet accounted for. We split this into two further subsets:
			\begin{itemize}
				\item The set $\CR^\sharp(G)$ consisting of all primes $p\in \CR(G)$ for which there exists an integer $k\ge0$ such that
					\[
					\frac{\mu(\CV_{p^k})}{\mu(\CV)}\ge 1-\frac{C_2}{p}\quad\text{and}\quad
					\frac{\mu(\CW_{p^k})}{\mu(\CW)} \ge 1-\frac{C_2}{p},
					\]
					and 
					\[
					q(G_{p^a,p^b})<M\cdot q(G)
					\quad\text{for all}\ (a,b)\in\Z_{\ge0}^2\ \text{with}\ a\neq b.
					\]
				\item The set $\CR^\flat(G):=\CR(G)\setminus \CR^\sharp(G)$.
			\end{itemize}
	\end{dfn}

	In the next section, we reduce the proof of Propositions \ref{prop:main-prop1}-\ref{prop:main-prop3}  (and hence  Theorem \ref{thm:variance}) to the existence of
GCD subgraphs with nice properties as described in the following three propositions and three lemmas.
	
	\begin{proposition}[Iteration when $\CR^\flat(G)\ne\emptyset$]\label{prop:IterationStep1}
		Let $G=(\mu,\CV,\CW,\CE,\CP,f,g)$ be a GCD graph with edge density $\delta>0$ such that 
		\[
		\CR(G)\subseteq\{p>C_6\}\quad\text{and}\quad \CR^\flat(G)\neq\emptyset.
		\] 
		Then there is a GCD subgraph $G'$ of $G$ with edge density $\delta'>0$ and multiplicative data $(\CP',f',g')$ such that:
		\begin{enumerate}[label=(\alph*)]
			\item $G'$ is $(\CP',f',g')$-maximal;
			\item $\CP\subsetneq\CP'\subseteq\CP\cup\CR(G)$;
			\item$\CR(G')\subsetneq \CR(G)$;
			\item $q(G')\geqslant M^N q(G)$ with $N=\#\{p\in\CP'\setminus \CP:f'(p)\neq g'(p)\}$;
			\item $p^{f'(p)}\|v$ and $p^{g'(p)}\| w$ for all $p\in \CP'\setminus \CP$, $(v,w)\in \CV\times \CW$.
		\end{enumerate}
	\end{proposition}

	\begin{proposition}[Iteration when $\CR^\flat(G)=\emptyset$]\label{prop:IterationStep2}
		Let $G=(\mu,\CV,\CW,\CE,\CP,f,g)$ be a GCD graph with edge density $\delta>0$ such that 
		\[
		\CR(G)\subseteq\{p> C_6\},
		\quad
		\CR^\flat(G)=\emptyset ,
		\quad
		\CR^\sharp(G)\neq\emptyset.
		\] 
		Then there is a GCD subgraph $G'=(\mu,\CV',\CW',\CE',\CP',f',g')$ of $G$ such that
		\[
		\CP\subsetneq\CP'\subseteq\CP\cup\CR(G),
		\quad
		\CR(G')\subsetneq \CR(G),
		\quad 
		q(G')\geqslant q(G).
		\]
	\end{proposition}

	\begin{proposition}[Bounded quality loss for small primes]\label{prop:SmallPrimes}
		Let  $G=(\mu,\mathcal{V},\mathcal{W},\mathcal{E},\mathcal{P},f,g)$  be a GCD graph with edge density $\delta>0$. Then there is a GCD subgraph $G'=(\mu,\CV',\CW',\CE',\CP',f',g')$ of $G$ with edge density $\delta'>0$ such that
		\[
		\mathcal{P}' \subseteq \mathcal{P} \cup \bg( \mathcal{R}(G) \cap \{ p\leqslant C_6 \}\bg),
		\quad
		\CR(G')\subseteq\{p> C_6\},
		\quad
	q(G')\ge q(G)/C_7.
		\]
	\end{proposition}

	\begin{lemma}[Removing the effect of $\CR(G)$ from $L_t(v,w)$]\label{lem:Cosmetic-L}
		Let $G=(\mu,\CV,\CW,\CE,\CP,f,g)$ be a maximal GCD graph with edge density $\delta>0$ and let $s\ge1$ and $t\ge C_6s$. Assume further that
		\[
		\CR^\flat(G)=\emptyset
		\quad\text{and}\quad
		\CE\subseteq \bg\{(v,w)\in\CV\times\CW: L_t(v,w)\ge 1/s \bg\}.
		\]
		Then there exists a GCD subgraph $G'=(\mu,\CV,\CW,\CE',\CP,f,g)$ of $G$ such that
		\[
		q(G')\ge \frac{q(G)}{2}>0\quad\text{and}\quad
		\CE'\subseteq 
		\bggg\{(v,w)\in\CV\times\CW:
		\sum_{\substack{p|v w/\gcd(v,w)^2\\p> t,\ p\notin \CR(G)}}\frac{1}{p}\ge \frac{1}{2s} \bggg\}.
		\]
	\end{lemma}
	
		\begin{lemma}[Removing the effect of $\CR(G)$ from $\omega_t(v,w)$]\label{lem:Cosmetic-omega}
		Let $G=(\mu,\CV,\CW,\CE,\CP,f,g)$ be a maximal GCD graph with edge density $\delta>0$ and let $t\ge \exp\exp(C_6)$ and $U\ge 2C_6\log\log t$. Assume further that
		\[
		\CR^\flat(G)=\emptyset
		\quad\text{and}\quad
		\CE\subseteq \bg\{(v,w)\in\CV\times\CW: \omega_t(v,w)\ge U \bg\}.
		\]
		Then there exists a GCD subgraph $G'=(\mu,\CV,\CW,\CE',\CP,f,g)$ of $G$ such that
		\[
		q(G')\ge \frac{q(G)}{2}>0
		\quad\text{and}\quad
		\CE'\subseteq 
		\bggg\{(v,w)\in\CV\times\CW:
		\sum_{\substack{p|v w/\gcd(v,w)^2\\C_6<p\le t,\ p\notin \CR(G)}} 1  \ge \frac{U}{2} \bggg\} .
		\]
	\end{lemma}

	\begin{lemma}[Subgraph with high-degree vertices]\label{lem:HighDegreeSubgraph}
		Let $G=(\mu,\CV,\CW,\CE,\CP,f,g)$ be a GCD graph with edge density $\delta>0$. Then there is a GCD subgraph $G'=(\mu,\CV',\CW',\CE',\CP,f,g)$ of $G$ with edge density $\delta'>0$ such that:
		\begin{enumerate}[label=(\alph*)]
			\item $q(G')\geqslant q(G)$;
			\item For all $v\in\CV'$ and for all $w\in\CW'$, we have
			\[
			\mu(\Gamma_{G'}(v))\geqslant\frac{1+\tau}{2+\tau}\cdot \delta'\mu(\CW')
			\quad\text{and}\quad 
			\mu(\Gamma_{G'}(w))\geqslant \frac{1+\tau}{2+\tau}\cdot \delta' \mu(\CV').
			\]
		\end{enumerate}
		
	\end{lemma}
	
\section{Proof of Propositions \ref{prop:main-prop1}-\ref{prop:main-prop3} assuming the results of Section \ref{sec:GCDgraphs}}\label{sec:proof-of-moments-bound}
	
\begin{dfn}[Neighbourhood sets]\label{dfn:neighborhood sets} Given a bipartite graph $G=(\CV,\CW,\CE)$, we define the \emph{neighbourhood sets} by
	\[
	\Gamma_G(v):=\{w\in\CW:\,(v,w)\in\CE\}\quad\text{for any}\ v\in\CV,
	\]
	and 
	\[
	\Gamma_G(w):=\{v\in\CV:\,(v,w)\in\CE\}
	\quad\text{for any}\ w\in \CW.
	\]			
\end{dfn}

\begin{proof}[Proof of Proposition \ref{prop:main-prop1}]
Let $\CS=\Z\cap[1,Q]$, $\mu(q)=\varphi(q)\psi(q)/q$ and $\CE=\{(q,r)\in\CS^2:D(q,r)\le y\}$, and consider the GCD graph $G=(\mu,\CS,\CS,\mathcal{E},\emptyset,f_\emptyset,g_\emptyset)$ with trivial set of primes. Let $\delta=\delta(G)$ be its edge density. It suffices to show that
	\begin{equation}
		\label{eq:main-prop1 rephrasing}
			q(G) \ll_\tau y^2.
	\end{equation}
	Indeed, we have that 
	\[
	q(G)=\delta^{2+\tau}\mu(\CS)^2 = \mu(\CE)^{2+\tau} \mu(\CS)^{-2-2\tau}
	\]
	and $\mu(\CS)= \frac{1}{2}  \Psi(Q)$. Hence, if \eqref{eq:main-prop1 rephrasing} holds, then 
	\[
	\mu(\CE) \ll_\tau y^{\frac{2}{2+\tau}}\Psi(Q)^{\frac{2+2\tau}{2+\tau}}.
	\]
	This proves Proposition \ref{prop:main-prop1} with $\eps=\tau/(2+\tau)$. Since $\tau>0$ can be taken to be arbitrarily small, this completes the proof of Proposition \ref{prop:main-prop1}. 
	
	Let us now prove \eqref{eq:main-prop1 rephrasing}. If $\delta=0$, relation \eqref{eq:main-prop1 rephrasing} holds trivially, so let us assume that $\delta>0$. We use the results of Section \ref{sec:GCDgraphs} with $M=2$. We first use Proposition \ref{prop:SmallPrimes}, and we then iterate Propositions \ref{prop:IterationStep1} and \ref{prop:IterationStep2}. Lastly, we apply Lemma \ref{lem:HighDegreeSubgraph}. Thus, we find that there exists a GCD subgraph $G' = (\mu,\mathcal{V}',\mathcal{W}',\mathcal{E}',\mathcal{P}',f',g')$ of $G$ of density $\delta'>0$. such that:
		\begin{enumerate}[label=(\alph*)]
			\item $\mathcal{R} (G') = \emptyset$;
			\item For all $v \in \mathcal{V}'$, we have $\mu(\Gamma_{G'}(v)) \geqslant \frac{1+ \tau }{2+\tau} \delta' \mu(\mathcal{W}')\ge 0.5\delta'\mu(\CW')$;
			\item For all $w \in \mathcal{W}'$, we have $\mu(\Gamma_{G'}(w)) \geqslant \frac{1+ \tau }{2+\tau} \delta'\mu(\mathcal{V}')\ge 0.5\delta'\mu(\CV')$;
			\item $q(G') \gg_\tau q(G)$.
		\end{enumerate}
	Hence, it suffices to prove that
		\begin{equation}
		\label{eq:main-prop1 rephrasing2}
		q(G') \ll_\tau y^2.
	\end{equation}
	To do this, we largely follow the proof of Proposition 6.3 in \cite{KM}, with some important simplifications because we do not need to keep track of the `anatomical' condition $L_t(v,w)\ge 10$ as we did there. For the ease of the reader, we repeat all the details instead of referring to \cite{KM}.
		
	Set
	\begin{equation}
		\label{eq:ab1}
	a:=\prod_{p\in\CP'}p^{f'(p)}
	\quad\text{and}\quad
	b:=\prod_{p\in\CP'}p^{g'(p)}.
	\end{equation}
	The definition of a GCD graph implies that 
	\begin{equation}
		\label{eq:ab2}
	a|v\quad\text{for all}\ v\in\CV',
	\qquad b|w\quad \text{for all}\ w\in\CW'.
	\end{equation}
	Moreover, since $\CR(G')=\emptyset$, and $p^{\min\{f'(p),g'(p)\}}\|\gcd(v,w)$ for all $(v,w)\in\CE'$, we have that 
	\begin{equation}
		\label{eq:ab3}
	\gcd(v,w)=\gcd(a,b)
	\quad\text{for all}\quad(v,w)\in\CE'.
	\end{equation}
	
	Now, note that
	\[
	\prod_{p\in\CP'}p^{|f'(p)-g'(p)|} = \prod_{p\in\CP'}p^{\max\{f'(p),g'(p)\}-\min\{f'(p),g'(p)\}}= \frac{\lcm(a,b)}{\gcd(a,b)}
	= \frac{ab}{\gcd(a,b)^2} ,
	\]
	as well as 
	\[
	\prod_{p\in\CP'}\frac{1}{(1-\un_{f'(p)=g'(p)\ge1}/p)^2(1-1/p^{1+\frac{\tau}{4}})^3} 
	\ll_\tau \prod_{p\in\CP'}\frac{1}{(1-\un_{f'(p)=g'(p)\ge1}/p)^2} \le \frac{ab}{\phi(a)\phi(b)} .
	\]
	Consequently, from the definition of $q(\cdot)$, we find
	\begin{align}
		q(G') &=  (\delta')^{2+\tau}\mu(\CV')\mu(\CW')
		\prod_{p\in\CP'} \frac{p^{|f'(p)-g'(p)|}}{(1-\un_{f'(p)=g'(p)\ge1}/p)^2(1-1/p^{1+\frac{\tau}{4}})^3}\nonumber\\
		& \ll_\tau (\delta')^{2+\tau} \mu(\CV')\mu(\CW') \cdot \frac{ab}{\gcd(a,b)^2} \cdot  \frac{ab}{\phi(a)\phi(b)}\nonumber \\
		&= (\delta')^{1+\tau} \mu(\CE')\cdot  \frac{ab}{\gcd(a,b)^2} \cdot  \frac{ab}{\phi(a)\phi(b)} .
		\label{eq:QG'Bound}
	\end{align}
	It remains to bound $\mu(\CE')$.
	
	We have that
	\[
	\CE'\subseteq\CE=\{(v,w)\in\CS\times\CS: D(v,w)\le y \},
	\]
	where we recall that $D(v,w)=\max\{v\psi(w),w\psi(v)\}/\gcd(v,w)$. 
	Since $\gcd(v,w)=\gcd(a,b)$ for all $(v,w)\in\CE'$, we infer that
	\[
	\psi(v)\le \frac{y\cdot \gcd(a,b)}{w}
	\quad\text{and}\quad 
	\psi(w)\le \frac{y\cdot \gcd(a,b)}{v}
	\quad\text{for all}\quad 
	(v,w)\in\CE'.
	\]
	The vertex sets $\CV',\CW'$ are finite sets of positive integers. For each $v\in\CV'$, let $w_{\max}(v)$ be the largest integer in $\CW'$ such that $(v,w_{\max}(v))\in\CE'$. (This quantity is well-defined in virtue of property (b) above. In addition, we emphasise to the reader that `largest' refers to the size of elements as positive integers, and does not depend on the measure $\mu$.) Similarly, for each $w\in\CW'$, let $v_{\max}(w)$ be the largest element of $\CV'$ such that $(v_{\max}(w),w)\in\CE'$. Consequently,
	\begin{equation}\label{psi ub}
		\psi(v)\le \frac{y\cdot \gcd(a,b)}{w_{\max}(v)}
		\quad\text{and}\quad 
		\psi(w)\le \frac{y\cdot \gcd(a,b)}{v_{\max}(w)}
		\quad\text{for all}\quad v\in \CV',\,w\in\CW'.
	\end{equation}

	Now, let $w_0$ be the largest integer in $\CW'$ and $\CE''=\{(v,w)\in\CE': (v,w_0)\in\CE'\}$. We then have
	\begin{equation}\label{eq:w0}
		w_{\max}(v)=w_0 \quad\text{for all}\quad(v,w)\in\CE'' .
	\end{equation}
	In addition, by properties (b) and (c) of $G'$, we have
	\[
			\mu(\CE'') 
		=\sum_{v\in \Gamma_{G'}(w_0)}\mu(v)\mu(\Gamma_{G'}(v))
		\ge \mu(\Gamma_{G'}(w_0)) \cdot \frac{\delta' \mu(\CW')}{2}		
		\ge \bigg(\frac{\delta'}{2} \bigg)^2 \mu(\CV')\mu(\CW') = \frac{\delta'\mu(\CE')}{4} .
	\]
	Substituting this bound into \eqref{eq:QG'Bound}, we find
	\begin{equation}\label{eq:q(G') bound}
		q(G') \ll_\tau (\delta')^{\tau} \mu(\CE'') \cdot  \frac{a b}{\phi(a)\phi(b)}\cdot \frac{a b}{\gcd(a,b)^2} 
		\le \mu(\CE'')  \cdot \frac{ab}{\gcd(a,b)^2} \cdot \frac{ab}{\phi(a)\phi(b)} ,
	\end{equation}
	where we used the trivial bound $\delta'\le 1$ in the second inequality and our assumption that $\tau>0$.
	
	In addition,
	\als{
		\mu(\CE'')  = \sum_{(v,w)\in\CE''} \frac{\psi(v)\phi(v)}{v} \cdot \frac{\psi(w)\phi(w)}{w} .
	}
	Since $a|v$ and $b|w$, we have $\phi(v)/v\le \phi(a)/a$ and $\phi(w)/w\le \phi(b)/b$. Therefore 
	\[
	\mu(\CE'')\le \frac{\phi(a)\phi(b)}{a b}\sum_{(v,w)\in\CE''}\psi(v)\psi(w).
	\]
	Together with \eqref{psi ub}, \eqref{eq:w0} and \eqref{eq:q(G') bound}, this implies that
	\begin{equation}
		q(G') \ll  y^2ab \sum_{(v,w)\in\CE''} \frac{1}{v_{\max}(w) w_0}  
		\le  y^2ab \sum_{(v,w)\in\CE'} \frac{1}{v_{\max}(w) w_0}  .
		\label{eq:G'Quality}
	\end{equation}
	Writing $v=v' a$ and $w=w' b$, we find that
	\begin{align*}
		\sum_{(v,w)\in\CE'} \frac{1}{v_{\max}(w) w_0}
		&\le \sum_{w'\le w_0/b}\frac{1}{ w_0 v_{\max}(b w')}
		\sum_{v'\le v_{\max}(b w')/a}1\\
		&\le \sum_{w'\le w_0/b}\frac{1}{ w_0 v_{\max}(b w')}\cdot \frac{v_{\max}(bw')}{a} \\
		&\le \frac{1}{ab} .
	\end{align*}
	Together with \eqref{eq:G'Quality}, this implies that
	\[
	q(G')\ll y^2 .
	\]
	as needed. This completes the proof of Proposition \ref{prop:main-prop1}.
\end{proof}

\begin{proof}[Proof of Proposition \ref{prop:main-prop2}] Let $C_6$ be defined as in \eqref{eq:CDefs} , where the parameter $M$ will be determined later. If $t\le C_6s$, then the result follows immediately by Proposition \ref{prop:main-prop1}. Let us assume now that $t\ge C_6s$.
	
Let $\CS=\Z\cap[1,Q]$, $\mu(q)=\varphi(q)\psi(q)/q$ and $\CE=\{(q,r)\in\CS^2:D(q,r)\le y,\ L_t(q,r)\ge 1/s\}$, and consider the GCD graph $G=(\mu,\CS,\CS,\mathcal{E},\emptyset,f_\emptyset,g_\emptyset)$ with trivial set of primes. Let $\delta=\delta(G)$ be its edge density. As in the proof of Proposition \ref{prop:main-prop1}, it suffices to show that
	\begin{equation}
		\label{eq:main-prop2 rephrasing}
		q(G) \ll_{\tau,C} y^2 e^{-Ct/s}. 
	\end{equation}
We may assume that $\delta>0$; otherwise, \eqref{eq:main-prop2 rephrasing} holds trivially. To proceed, we adapt the argument leading to Propositions 6.3 and 7.1 of \cite{KM}. 

We shall apply the results of Section \ref{sec:GCDgraphs} on GCD graphs with the parameter $M=e^{4C}$. First, we use Proposition \ref{prop:SmallPrimes} to find a GCD subgraph $G^{(0)} = (\mu,\mathcal{V}^{(0)},\mathcal{W}^{(0)},\mathcal{E}^{(0)},\mathcal{P}^{(0)},f^{(0)},g^{(0)})$ of $G$ of density $\delta^{(0)}>0$ such that:
	\begin{enumerate}[label=(\alph*)]
	\item $\CP^{(0)}\subseteq\{p\le C_6\}$;
	\item $\CR(G^{(0)})\subseteq\{p>C_6\}$;
	\item $q(G^{(0)}) \gg_{\tau,C} q(G)$.
\end{enumerate}

Next, we claim we may find a GCD subgraph $G^{(1)} = (\mu,\mathcal{V}^{(1)},\mathcal{W}^{(1)},\mathcal{E}^{(1)},\mathcal{P}^{(1)},f^{(1)},g^{(1)})$ of $G^{(0)}$ that has density $\delta^{(1)}>0$ and that satisfies the following properties:
\begin{enumerate}[label=(\alph*)]
	\item $G^{(1)}$ is maximal;
	\item $\CR^\flat(G^{(1)}) = \emptyset$;
	\item $q(G^{(1)}) \ge e^{4CN} q(G^{(0)})$ 
	with $N=\#\{p\in\CP^{(1)}: p>C_6,\ f^{(1)}(p)\neq g^{(1)}(p)\}$;
	\item $p^{f^{(1)}(p)}\| v$ and $p^{g^{(1)}(p)}\|w$ for all $(v,w)\in \CV^{(1)}\times \CW^{(1)}$ and all primes $p>C_6$ lying in $\CP^{(1)}$. 
\end{enumerate}

Indeed, if $\CR^\flat(G^{(0)})\neq\emptyset$, we may apply 
Proposition \ref{prop:IterationStep1} in an iterative fashion to construct $G^{(1)}$. On the other hand, if $\CR^\flat(G^{(0)})=\emptyset$, then we need to be a bit careful because $G^{(0)}$ is not necessarily maximal, so property (a) above might not be automatically satisfied if we take $G^{(1)}=G^{(0)}$. In this case, we begin by first selecting a maximal GCD subgraph of $G^{(0)}$ with the same multiplicative data as $G^{(0)}$, let us call it $G'$. We then check whether the set $\CR^\flat(G')$ is empty or not. If it is empty, then we take $G^{(1)}=G'$ and check easily that this choice satisfies all four properties listed above (property (c) holds with $N=0$). Otherwise, if $\CR^\flat(G')\neq\emptyset$, then we repeatedly apply Proposition \ref{prop:IterationStep1} till we arrive to a GCD subgraph $G^{(1)}$ of $G'$ satisfying all four properties above. 

In all cases, we have constructed the claimed GCD graph $G^{(1)}$. We then separate two cases.
	
	\bigskip
	
	\noindent {\it Case 1:} $N\ge 0.25t/s$. We then iterate Propositions \ref{prop:IterationStep1} and \ref{prop:IterationStep2}, and we subsequently apply Lemma \ref{lem:HighDegreeSubgraph} to arrive at 
	a GCD subgraph $G' = (\mu,\mathcal{V}',\mathcal{W}',\mathcal{E}',\mathcal{P}',f',g')$ of $G^{(1)}$ of density $\delta'>0$ such that:
	\begin{enumerate}[label=(\alph*)]
		\item $\mathcal{R} (G') = \emptyset$;
		\item For all $v \in \mathcal{V}'$, we have $\mu(\Gamma_{G'}(v)) \geqslant \frac{1+ \tau }{2+\tau} \delta' \mu(\mathcal{W}')\ge 0.5\delta'\mu(\CW')$;
		\item For all $w \in \mathcal{W}'$, we have $\mu(\Gamma_{G'}(w)) \geqslant \frac{1+ \tau }{2+\tau} \delta'\mu(\mathcal{V}')\ge 0.5\delta'\mu(\CV')$;
		\item $q(G') \ge q(G^{(1)})$.
	\end{enumerate}
	In particular, we see that $q(G')\gg_\tau e^{4CN} q(G)\ge e^{Ct/s}q(G)$. In addition, arguing as in the proof of Proposition \ref{prop:main-prop1}, we have that $q(G')\ll_\tau y^2$. Hence, \eqref{eq:main-prop2 rephrasing} holds in this case.
	
	\bigskip
	
	\noindent {\it Case 2:} $N< 0.25t/s$. In this case we must exploit the condition $L_t(q,r)\ge 1/s$ to obtain the necessary savings. If we blindly apply Proposition \ref{prop:IterationStep2}, we might lose track of this condition. So we first perform some cosmetic surgery to our graph. Applying Lemma \ref{lem:Cosmetic-L}, we find that there exists a GCD subgraph $G^{(2)}=(\mu,\CV^{(2)},\CW^{(2)},\CE^{(2)},\CP^{(2)},f^{(2)},g^{(2)})$  of $G^{(1)}$ with 
	\begin{align}
		\CP^{(2)}&=\CP^{(1)},\label{eq:P3b}\\
		q(G^{(2)})&\ge \frac{q(G^{(1)})}{2},
		\label{eq:G3Quality}
	\end{align}
	and such that 
	\begin{equation}
		\sum_{\substack{p|v w/\gcd(v,w)^2\\p>t,\ p\notin \CR(G^{(1)})}}\frac{1}{p}\ge \frac{1}{2s}
		\quad\text{whenever}\quad
		(v,w)\in\CE^{(2)}.
		\label{eq:Anat1}
	\end{equation}
	
	We claim that an inequality of the form \eqref{eq:Anat1} holds even if we remove from consideration the primes lying in $\CP^{(1)}$. Indeed, if $p>C_6$ lies in $\CP^{(1)}$, $(v,w)\in \CE^{(2)}\subseteq\CE^{(1)}$  and  $ p | \frac{vw}{\gcd(v, w)^2}$, then we use property (d) of $G^{(1)}$ to find that $f^{(1)}(p)\neq g^{(1)}(p)$. Hence,
	\[
	\sum_{\substack{p|v w/\gcd(v,w)^2\\p>t,\ p\in \CP^{(1)} }}\frac{1}{p}
	< \frac{\#\{p\in \CP^{(1)}: f^{(1)}(p)\neq g^{(1)}(p),\ p>t\}} {t} \le \frac{N}{t}< \frac{1}{4s}
	\]
	by our assumption on $N$, and because $t\ge C_6s\ge C_6$ here. As a consequence, for all $(v,w)\in\CE^{(2)}$ we have
	\begin{equation}
		\sum_{\substack{p|v w/\gcd(v,w)^2\\ p>t,\ p\notin \CR(G^{(1)})\cup \CP^{(1)} }}\frac{1}{p}\ge\frac{1}{4s}. 
		\label{eq:Anat4}
	\end{equation}
	
	Next, we iterate Propositions \ref{prop:IterationStep1} and Proposition \ref{prop:IterationStep2}, and we also apply Lemma \ref{lem:HighDegreeSubgraph} to find a  GCD subgraph $G' = (\mu,\mathcal{V}',\mathcal{W}',\mathcal{E}',\mathcal{P}',f',g')$ of $G^{(2)}$ of density $\delta'>0$ such that:
	\begin{enumerate}[label=(\alph*)]
		\item $\mathcal{R} (G') = \emptyset$;
		\item For all $v \in \mathcal{V}'$, we have $\mu(\Gamma_{G'}(v)) \geqslant \frac{1+ \tau }{2+\tau} \delta' \mu(\mathcal{W}')\ge 0.5\delta'\mu(\CW')$;
		\item For all $w \in \mathcal{W}'$, we have $\mu(\Gamma_{G'}(w)) \geqslant \frac{1+ \tau }{2+\tau} \delta'\mu(\mathcal{V}')\ge 0.5\delta'\mu(\CV')$;
		\item $q(G') \ge q(G^{(2)})$;
		\item \eqref{eq:Anat4} holds for all $(v,w)\in \CE'$.
	\end{enumerate}
	Hence, Proposition \ref{prop:main-prop2} will follow in this case too if we show that
	\begin{equation}
	\label{eq:main-prop2 rephrasing2}
	q(G') \ll_{\tau,C} y^2e^{-Ct/s}. 
	\end{equation}
	To do this, we largely follow the proof of Propositions 6.3 and 7.1 in \cite{KM}. We give all details for the ease of the reader.

	As in the proof of Proposition \ref{prop:main-prop1} above, we define $a$ and $b$ by \eqref{eq:ab1} and take note of relations \eqref{eq:ab2} and \eqref{eq:ab3}. In addition, relation \eqref{eq:QG'Bound} implies that 
	\begin{equation}		
		\label{eq:QG'Bound again}
		q(G')  \ll_\tau (\delta')^{1+\tau} \mu(\CE')\cdot  \frac{ab}{\gcd(a,b)^2} \cdot  \frac{ab}{\phi(a)\phi(b)} .
	\end{equation}
	
	It remains to bound $\mu(\CE')$. First of all, note that
	\[
	\CP'
	\subseteq \CR(G^{(2)})\cup\CP^{(2)}
	\subseteq \CR(G^{(1)})\cup\CP^{(1)},
	\]
	where the second relation follows by \eqref{eq:P3b}. 
	Let $(v,w)\in\CE'$ and let us write $v=av'$ and $w=bw'$. Note that if $p|vw$ but $p\nmid v'w'$, then $p|ab$, and thus $p\in \CP'	\subseteq \CR(G^{(1)})\cup\CP^{(1)}$. Since $\gcd(v,w)=\gcd(a,b)$, we must have that $\gcd(v',w')=1$.  Together with \eqref{eq:Anat4}, this implies that 
	$L_t(v',w') \ge 1/(4s)$. We conclude that
	\[
	\CE'\subseteq \bggg\{(v,w)=(av',bw')\in\CV'\times\CW': 
	\begin{array}{l}
		\gcd(v,w)=\gcd(a,b),\ D(v,w)\le y,\\
		L_t(v',w')\ge \frac{1}{4s} 
	\end{array}\bggg\}.
	\]
	
	Next, let us define $w_{\max}(v)$, $v_{\max}(w)$, $w_0$ and $\CE''$ as in the proof of Proposition \ref{prop:main-prop1}. We find that 
	\begin{equation}
		q(G') \ll  y^2ab \sum_{(v,w)\in\CE''} \frac{1}{v_{\max}(w) w_0}  
		\le  y^2ab \sum_{(v,w)\in\CE'} \frac{1}{v_{\max}(w) w_0}  .
		\label{eq:G'Quality again}
	\end{equation}
	If $v=v' a$ and $w=w' b$ and $(v,w)\in\CE'$, then we know that $L_t(v',w')\ge1/(4s)$. In particular, we see that either 
	\[
	\sum_{p|v',\,p>t}\frac{1}{p}\ge \frac{1}{8s}
	\quad\text{or}
	\quad
	\sum_{p|w',\,p>t}\frac{1}{p}\ge \frac{1}{8s}
	\]
	whenever $(v,w)\in\CE'$. Consequently,
	\[
		\sum_{(v,w)\in\CE'} \frac{1}{v_{\max}(w) w_0}  
		\le S_1+S_2,
	\]
	where
	\begin{align*}
		S_1&\coloneqq \mathop{\sum\sum}_{\substack{w'\le w_0/b,\ v'\le v_{\max}(b w')/a\\ \sum_{p|v',\,p>t}1/p\ge 1/(8s)}}  \frac{1}{v_{\max}(b w') w_0} ,\\
		S_2&\coloneqq \mathop{\sum\sum}_{\substack{w'\le w_0/b,\ v'\le v_{\max}(b w')/a\\ \sum_{p|w',\,p>t}1/p\ge 1/(8s)}}  \frac{1}{v_{\max}(b w') w_0}  .
	\end{align*}
	For $S_1$, we note that 
	\begin{align*}
		S_1&\le \sum_{w'\le w_0/b} \frac{1}{w_0v_{\max}(b w')} \sum_{\substack{v'\le v_{\max}(b w')/a \\ \sum_{p|v',\,p>t}1/p\ge 1/(8s)}} 1\\
		&\ll_C \sum_{w'\le w_0/b} \frac{1}{w_0v_{\max}(b w')} \cdot \frac{v_{\max}(b w')/a}{e^{Ct/s}} \\
		&\le \frac{1}{ab e^{Ct/s}} 
	\end{align*}
	by Lemma \ref{lem:Anatomy-L}. Similarly for $S_2$, we find that
	\begin{align*}
		S_2&\le\sum_{\substack{w'\le w_0/b \\ \sum_{p|w',\,p>t}1/p\ge 1/(8s)}} \frac{1}{w_0v_{\max}(b w')} \sum_{v'\le v_{\max}(b w')/a}1\\
		&\le \sum_{\substack{w'\le w_0/b \\ \sum_{p|w',\,p>t}1/p\ge 1/(8s)}} \frac{1}{aw_0} \\
		&\ll_C \frac{1}{ab e^{Ct/s}} ,
	\end{align*}
	by applying Lemma \ref{lem:Anatomy-L} once again. Substituting these bounds into \eqref{eq:G'Quality again} establishes \eqref{eq:main-prop2 rephrasing2}, thus completing the proof of Proposition \ref{prop:main-prop2}.
\end{proof}

\begin{proof}[Proof of Proposition \ref{prop:main-prop3}] 
	We use a very similar argument to the one employed in the previous proof. We highlight only the important changes. 
	
	\begin{itemize}
		\item We may assume that $t$ is large enough in terms of $\kappa$ and $C$; otherwise, the result follows from Proposition \ref{prop:main-prop1}. 
		\item We use the results of Section \ref{sec:GCDgraphs} with $M=e^{4C/\kappa}$.
		\item We separate Cases 1 and 2 according to whether $N>0.25\kappa\log t$ or $N\le 0.25\kappa\log t$. 
		\item In Case 2, we aim to show that 
			\begin{equation}
			\label{eq:main-prop3 rephrasing}
			q(G') \ll_{\tau,\kappa,C} y^2 t^{-C} .
		\end{equation}
		\item In  Case 2, we use Lemma \ref{lem:Cosmetic-omega} with $U=\kappa\log t$ in place of Lemma \ref{lem:Cosmetic-L} to find $G^{(2)}$ such that $\#\{p|vw/\gcd(v,w)^2:C_6<p\le t,\ p\notin \CR(G^{(1)})\} \ge 0.5\kappa\log t$ whenever $(v,w)\in \CE^{(2)}$. 
		\item  In  Case 2, we note that $\#\{p|vw/\gcd(v,w)^2: p>C_6,\ p\in  \CP^{(1)}\}\le N\le 0.25\kappa\log t$, where we used property (d) of the graph $G^{(1)}$. Thus $\#\{p|vw/\gcd(v,w)^2: C_6<p\le t,\ p\notin \CR(G^{(1)})\cup\CP^{(1)}\}\ge 0.25\kappa\log t$.
		\item To prove \eqref{eq:main-prop3 rephrasing} and conclude the proof, we apply Lemma \ref{lem:Anatomy-omega} in place of \ref{lem:Anatomy-L}.
	\end{itemize}
	We leave the details for how to complete the proof of Proposition \ref{prop:main-prop3} to the reader.
\end{proof}

	\section{Proof of Lemma \ref{lem:HighDegreeSubgraph}}\label{sec:HighDegree}
	
	In this section we establish Lemma \ref{lem:HighDegreeSubgraph}. 
	
	Recall the definition of neighborhood sets (cf.~Definition \ref{dfn:neighborhood sets}). We then have the following lemma.
	
	\begin{lemma}[Quality increment or all vertices have high degree]\label{lem:HighDegree}
		Let	$G=(\mu,\CV,\CW,\CE,\CP,f,g)$ be a GCD graph with edge density $\delta>0$. 		Then one of the following holds:
		\begin{enumerate}[label=(\alph*)]
			\item For all $v\in\CV$ and for all $w\in\CW$, we have
			\[
			\mu(\Gamma_G(v))\geqslant\frac{1+\tau}{2+\tau} \cdot \delta\mu(\CW)
			\quad\text{and}\quad 
			\mu(\Gamma_G(w))\geqslant  \frac{1+\tau}{2+\tau} \cdot \delta \mu(\CV).
			\]
			\item There is a GCD subgraph $G'=(\mu,\CV',\CW',\CE',\CP,f,g)$ of $G$ with quality $q(G')> q(G)$.
		\end{enumerate}
	\end{lemma}

	\begin{proof} Assume that (a) fails. Then either its first or its second inequality fails. Assume that the first one fails for some $v\in\CV$; the other case is entirely analogous. Let $\CE'$ be the set of edges between the vertex sets $\CV\setminus\{v\}$ and $\CW$. Note that 
		\[
		\mu(\CE')= \mu(\CE)-\mu(v)\mu(\Gamma_G(v)) > \bggg(1 - \frac{1+\tau}{2+\tau}\bggg)\mu(\CE) > 0
		\]
		because $\mu(\Gamma_G(v))< \frac{1+\tau}{2+\tau} \delta\mu(\CW)$, $\mu(v)\leqslant \mu(\CV)$, and $\mu(\CE)>0$ by the assumption that $\delta>0$. In particular, we have $\mu(\CW),\mu(\CV\setminus\{v\})>0$. We then consider $G'=(\mu,\CV\setminus\{v\},\CW,\CE',\CP,f,g)$, which is a GCD subgraph of $G$. Let $G'$ have edge density $\delta'$. We claim that $\delta'\geqslant \delta$ and $q(G')\geqslant q(G)$.
		
		Indeed, we have
		\begin{align*}
			\mu(\CE')
			= \mu(\CE)-\mu(v)\mu(\Gamma_G(v)) 
			&> \delta \mu(\CV)\mu(\CW)-\frac{(1+\tau)\delta}{2+\tau}\mu(v)\mu(\CW) \\
			&=\delta \big(\mu(\CV)-\mu(v)\big)\mu(\CW)\cdot\Big(1+\frac{\mu(v)}{\mu(\CV)-\mu(v)}\cdot\frac{1}{2+\tau}\Big).
		\end{align*}
		Thus the edge density $\delta'$ of $G'$ satisfies
		\[
		\delta'=\frac{\mu(\CE')}{\mu(\CV\setminus\{v\})\mu(\CW)}
		=\frac{\mu(\CE')}{\big(\mu(\CV)-\mu(v)\big)\mu(\CW)}
		> \delta\cdot \Bigl(1+\frac{\mu(v)}{\mu(\CV)-\mu(v)}\cdot\frac{1}{2+\tau}\Bigr).
		\]
		Consequently,
		\begin{align*}
			(\delta')^{2+\tau}\mu(\CV\setminus\{v\})\mu(\CW)
			&> \delta^{2+\tau}\big(\mu(\CV)-\mu(v)\big)\mu(\CW)
			\Bigl(1+\frac{\mu(v)}{\mu(\CV)-\mu(v)}\Bigr)
			= \delta^{2+\tau} \mu(\CV)\mu(\CW) .
		\end{align*}
		This proves the claim that $q(G')>q(G)$, thus completing the proof of the lemma.
	\end{proof}
	
	\begin{proof}[Proof of Lemma \ref{lem:HighDegreeSubgraph}] Let $G'$ be a $(\CP,f,g)$-maximal GCD subgraph of $G$, and let us apply Lemma \ref{lem:HighDegree} to $G'$. By its maximality, conclusion (a) of  Lemma \ref{lem:HighDegree} must hold. This completes the proof. 
	\end{proof}

	\section{Lemmas on GCD graphs}\label{sec:Prep}
	
	\begin{lemma}[Quality variation for special GCD subgraphs]\label{lem:InducedGraphs}
		Let $G=(\mu,\CV,\CW,\CE,\CP,f,g)$ be a GCD graph, $p\in\CR(G)$ and $k,\ell\in\Z_{\ge0}$.  In addition, if $G$ is non-trivial and $\mu(\CV_{p^k}),\mu(\CW_{p^\ell})>0$, then we have
		\[
		\frac{q(G_{p^k,p^\ell})}{q(G)}
		=\bigg(\frac{\mu(\CE_{p^k,p^\ell})}{\mu(\CE)}\bigg)^{2+\tau}
		\bigg(\frac{\mu(\CV)}{\mu(\CV_{p^k})}\bigg)^{1+\tau}
		\bigg(\frac{\mu(\CW)}{\mu(\CW_{p^\ell})}\bigg)^{1+\tau}
		\frac{p^{|k-\ell|}}{(1-\un_{k=\ell\geqslant 1}/p)^2(1-1/p^{1+\tau/4})^{3}}.
		\]
	\end{lemma}
	\begin{proof}
		This follows from directly the definitions.
	\end{proof}

	\begin{lemma}[One subgraph must have limited quality loss]\label{lem:Pigeonhole}
		Let $G=(\mu,\CV,\CW,\CE,\CP,f,g)$ be a GCD graph with edge density $\delta>0$, and let $\CV=\CV_1\sqcup\dots \sqcup\CV_I$ and $\CW=\CW_1\sqcup\dots \sqcup\CW_J$ be partitions of $\CV$ and $\CW$. Then there is a GCD subgraph $G'=(\mu,\CV',\CW',\CE',\CP,f,g)$ of $G$ with edge density $\delta'>0$ such that
		\[
		q(G')\geqslant \frac{q(G)}{(I J)^{2+\tau}},\qquad \delta'\geqslant\frac{\delta}{I J},
		\]
		and with $\CV'\in\{\CV_1,\dots,\CV_I\}$, $\CW'\in\{\CW_1,\dots,\CW_J\}$, and $\CE'=\CE(\CV',\CW')$.
	\end{lemma}
	\begin{proof}
		This is Lemma 11.2 in \cite{KM} with the obvious changes to account for the modified definition of $q(G)$. 
	\end{proof}

	\begin{lemma}[Few edges between unbalanced sets, I]\label{lem:UnbalancedSetEdges1}
		Let	$G=(\mu,\CV,\CW,\CE,\CP,f,g)$ be a GCD graph with edge density $\delta>0$. Let $p\in\CR(G)$, $r\in\Z_{\geqslant1}$ and $k\in\mathbb{Z}_{\geqslant 0}$ be such that $p^r> C_4$ and
		\[
		\frac{\mu(\CW_{p^k})}{\mu(\CW)}\geqslant 1-\frac{C_2}{p}.
		\]
		(In particular, if $p\leqslant C_2$, the last hypothesis is vacuous.) 
		
		If we set $\CL_{k,r}=\{\ell\in\Z_{\geqslant0}:|\ell-k|\geqslant r+1\}$, then one of the following holds:
		\begin{enumerate}[label=(\alph*)]
			\item There is $\ell\in\CL_{k,r}$ such that $q(G_{p^k,\,p^\ell})>M\cdot q(G)$.
			\item $\sum_{\ell\in\CL_{k,r}}  \mu(\CE_{p^k,\,p^\ell})\leqslant\mu(\CE)/(4p^{1+\tau/4})$.
		\end{enumerate}
	\end{lemma}
	
	\begin{proof} Assume that conclusion $(b)$ does not hold, so $\sum_{\ell\in\CL_{k,r}}  \mu(\CE_{p^k,p^\ell})>\mu(\CE)/(4p^{1+\tau/4})$ and we wish to establish $(a)$. Then there must exist some $\ell\in\CL_{k,r}$ such that 
		\[
		\mu(\CE_{p^k,p^\ell}) > \frac{\mu(\CE)}{300\cdot2^{|k-\ell|/20}p^{1+\tau/4}} >0 ,
		\]
		where we used that $\sum_{|j|\geqslant 0}2^{-|j|/20}\leqslant 2/(1-2^{-1/20})\leqslant 60$. In particular, $G_{p^k,p^\ell}$ is a non-trivial GCD graph. Since $\mu(\CW_{p^k})\geqslant (1-C_2/p)\mu(\CW)$ and $\ell\neq k$, we have that $\mu(\CW_{p^\ell})\leqslant C_2\mu(\CW)/p$. Consequently,
		\begin{align*}
			\frac{q(G_{p^k,p^\ell})}{q(G)}
			&= \bigg(\frac{\mu(\CE_{p^k,p^\ell})}{\mu(\CE)}\bigg)^{2+\tau}
			\bigg(\frac{\mu(\CV)}{\mu(\CV_{p^k})}\bigg)^{1+\tau}
			\bigg(\frac{\mu(\CW)}{\mu(\CW_{p^\ell})}\bigg)^{1+\tau} \frac{p^{|k-\ell|}}{(1-1/p^{1+\tau/4})^3} \\
			&\geqslant \Big(\frac{1}{300\cdot 2^{|k-\ell|/20}p^{1+\tau/4}}\Big)^{2+\tau}
			\bggg(\frac{p}{C_2}\bggg)^{1+\tau} p^{|k-\ell|} \\
			&\geqslant \frac{1}{300^{2+\tau}C_2^{1+\tau}}\cdot \bggg( \frac{p}{2^{\frac{2+\tau}{20}}}\bggg)^{|k-\ell|} \cdot p^{-1 - \tau/2 - \tau^2/4 }.
		\end{align*}
		
		Note that $0< \tau \leqslant 0.01$, so $1 + \tau/2 + \tau^2/4 \leqslant 1.01$ and $ (2+\tau)/20\leqslant 1/9$. Since $|k-\ell|\geqslant r+1\geqslant 0.99r + 1.01$ and $p/2^{1/9} \geqslant p^{8/9}$, we have 
		\[
		\bggg(\frac{p}{2^{\frac{2+\tau}{20}}}\bggg)^{|k-\ell|} \cdot p^{-1 - 0.5\tau - 0.25 \tau^2 } \geqslant p^{\frac{88}{100}r}\cdot \frac{1}{2^{\frac{101}{900}}}.
		\]
		Therefore
		\begin{align*}
			\frac{q(G_{p^k,p^\ell})}{q(G)}	&\geqslant  \frac{1}{300^{2.01}C_2^{1.01}} \cdot p^{\frac{88}{100}r}\cdot \frac{1}{2^{\frac{101}{900}}}
			>M
		\end{align*}
		by our assumption that $p^r> C_4$ (recall \eqref{eq:CDefs}). 
%
%
		This completes the proof of the lemma.
	\end{proof}

	Clearly, the symmetric version of Lemma \ref{lem:UnbalancedSetEdges1}  also  holds:

	\begin{lemma}[Few edges between unbalanced sets, II]\label{lem:UnbalancedSetEdges2}
		Let $G=(\mu,\CV,\CW,\CE,\CP,f,g)$ be a GCD graph with edge density $\delta>0$. Let $p\in\CR(G)$,  $r\in\mathbb{Z}_{\geqslant 1}$ and $\ell\in\Z_{\geqslant0}$ be such that $p^r>C_4$ and
		\[
		\frac{\mu(\CV_{p^\ell})}{\mu(\CV)}\geqslant 1-\frac{C_2}{p},
		\]
		and set $\CK_{\ell,r}=\{k\in\Z_{\geqslant0}:|\ell-k|\geqslant r+1\}$. Then one of the following holds:
		\begin{enumerate}[label=(\alph*)]
			\item There is $k\in\CK_{\ell,r}$ such that $q(G_{p^k,\,p^\ell})>M\cdot q(G)$.
			\item $\sum_{k\in\CK_{\ell,r}}  \mu(\CE_{p^k,\,p^\ell})\leqslant\mu(\CE)/(4p^{1+\tau/4})$.
		\end{enumerate}
	\end{lemma}

	\begin{lemma}[Few edges between small sets]\label{lem:SmallSetEdges}
		Let	$G=(\mu,\CV,\CW,\CE,\CP,f,g)$ be a GCD graph with edge density $\delta>0$ and let $\eta\in(0,1]$. Then one of the following holds:
		\begin{enumerate}[label=(\alph*)]
			\item For all sets $\CA\subseteq\CV$ and $\CB\subseteq\CW$ such that $\mu(\CA)\leqslant \eta\cdot  \mu(\CV)$ and  $\mu(\CB)\leqslant\eta \cdot \mu(\CW)$, we have $\mu(\CE(\CA,\CB))\leqslant \eta^{(2+2\tau)/(2+\tau)}\cdot \mu(\CE)$.
			\item There is a GCD subgraph $G'=(\mu,\CV',\CW',\CE',\CP,f,g)$ of $G$ such that $q(G')>q(G)$,  $\CV'\subsetneq \CV$ and $\CW'\subsetneq \CW$.
		\end{enumerate}
	\end{lemma}
	
	\begin{proof} When $\eta=1$, case (a) holds trivially. When $\eta<1$, the claimed result is  Lemma 11.5 in \cite{KM} with the obvious changes to account for the modified definition of $q(G)$. 
	\end{proof}

	\begin{lemma}[Subgraph with few edges between all small sets]\label{lem:NoSmallSetEdges}
		Let	$G=(\mu,\CV,\CW,\CE,\CP,f,g)$ be a GCD graph with edge density $\delta>0$, and let $\eta\in(0,1]$. Then there is a GCD subgraph $G'=(\mu,\CV',\CW',\CE',\CP,f,g)$ of $G$ with edge density $\delta'>0$ such that both of the following hold:
		\begin{enumerate}[label=(\alph*)]
			\item$q(G')\geqslant q(G)>0$.
			\item For all sets $\CA\subseteq\CV'$ and $\CB\subseteq\CW'$ such that $\mu(\CA)\leqslant \eta\cdot \mu(\CV')$ and $\mu(\CB)\leqslant \eta\cdot \mu(\CW')$, we have $\mu(\CE'\cap(\CA\times\CB))\leqslant \eta^{(2+2\tau)/(2+\tau)}\mu(\CE')$.
		\end{enumerate}
	\end{lemma}
	
	\begin{proof} 
		By iterating  Lemma \ref{lem:SmallSetEdges}.
	\end{proof}

		\begin{lemma}\label{lem:bound on edge set for anatomy} 
		Let $G=(\mu,\CV,\CW,\CE,\CP,f,g)$ be a GCD graph with $\CR^\flat(G)=\emptyset$ and edge density $\delta>0$. Let $p>C_6$ be a prime in $\CR(G)$. There exists an integer $k\ge0$ such that if we let $\CV'=\CV\setminus \CV_{p^k}$ and $\CW'=\CW\setminus \CW_{p^k}$, then one of the following holds:
		\begin{enumerate}[label=(\alph*)]
			\item We have
			\[
			\mu\bgg( \CE\bg((\CV\times \CW')\cup (\CV'\times \CW)  \bg)\bgg)  \le \frac{C_6}{10^{10}p}\cdot \mu(\CE) .
			\]
		\item The GCD subgraph $G'=(\mu,\CV',\CW',\CE',\CP,f,g)$ of $G$ with $\CE'=\CE(\CV',\CW')$ satisfies $q(G')> q(G)$.
	\end{enumerate}
	\end{lemma}
	
	\begin{proof}
		Let $p>C_6$ be a prime in $\CR(G)$. Since $\CR^\flat(G)=\emptyset$, we must have $p\in\CR^\sharp(G)$. Therefore there exists an integer $k\ge0$ such that	
		\begin{equation}
			\label{eq:p^k divides almost all}
			\frac{\mu(\CV_{p^k})}{\mu(\CV)}\ge 1-\frac{C_2}{p}\quad\text{and}\quad
			\frac{\mu(\CW_{p^k})}{\mu(\CW)} \ge 1-\frac{C_2}{p},
		\end{equation}
		and 
		\begin{equation}
			\label{eq:no good quality increment}
			q(G_{p^a,p^b})<M\cdot q(G)
			\quad\text{for all}\ (a,b)\in\Z_{\ge0}^2\ \text{with}\ a\neq b.
		\end{equation}
		In particular, we have 
		\[
			\frac{\mu(\CV')}{\mu(\CV)}\le \frac{C_2}{p}\quad\text{and}\quad
		\frac{\mu(\CW')}{\mu(\CW)} \le \frac{C_2}{p}.
		\]
		By the proof of Lemma \ref{lem:SmallSetEdges} applied with $\eta=\min\{1,C_2/p\}$ we see that one of the following holds: 
		\begin{enumerate}
			\item $\mu(\CE')\le \mu(\CE) \cdot (C_2/p)^{\frac{2+2\tau}{2+\tau}}$;
			\item $q(G')>q(G)$.
		\end{enumerate}
		If (b) holds, we are done, so let us assume that (a) holds. In particular, we then find that
		\begin{equation}
			\label{eq:bound on edges from small sets}	
		\mathop{\sum\sum}_{a, b\in\Z_{\ge0}\setminus\{k\}} \mu(\CE_{p^a,p^b}) = \mu(\CE')
		\le \mu(\CE) \cdot \bggg(\frac{C_2^{2+2\tau}}{p^\tau}\bggg)^{\frac{1}{2+\tau}} \cdot \frac{1}{p} \le \mu(\CE) \cdot \frac{1}{10^{10}p} 
		\end{equation}
		since $p>C_6\ge C_2^{10/\tau}$ and $\tau\le 0.01$  (recall \eqref{eq:CDefs}). Moreover, since $p>C_6\ge C_4$, we may apply Lemmas \ref{lem:UnbalancedSetEdges1} and \ref{lem:UnbalancedSetEdges2} with $r=1$. Relation \eqref{eq:no good quality increment} tells us that we must in case (b) of these two lemmas. We must therefore have that
		\begin{equation}
			\label{eq:bound on remote edge sets}
			\sum_{\substack{b\ge0\\ |k-b|\ge2}} \mu(\CE_{p^k,p^b}) \le \frac{\mu(\CE)}{4p}
			\quad\text{and}\quad
			\sum_{\substack{a\ge0\\ |a-k|\ge2}} \mu(\CE_{p^a,p^k}) \le \frac{\mu(\CE)}{4p} .
		\end{equation}
		In addition, if we take $(a,b)=(k,k+1)$ in \eqref{eq:no good quality increment} and we use Lemma \ref{lem:InducedGraphs}, we find that
		\[
		M>\frac{q(G_{p^k,p^{k+1}})}{q(G)}
		\ge \bigg(\frac{\mu(\CE_{p^k,p^{k+1}})}{\mu(\CE)}\bigg)^{2+\tau}
		\bigg(\frac{\mu(\CV)}{\mu(\CV_{p^k})}\bigg)^{1+\tau}
		\bigg(\frac{\mu(\CW)}{\mu(\CW_{p^{k+1}})}\bigg)^{1+\tau} p .
		\]
		Using \eqref{eq:p^k divides almost all}, we find that $\mu(\CW)/\mu(\CW_{p^{k+1}})\ge p/C_2$. In addition, we have the trivial lower bound $\mu(\CV)/\mu(\CV_{p^k})\ge1$. Therefore we conclude that
		\[
		M>\bigg(\frac{\mu(\CE_{p^k,p^{k+1}})}{\mu(\CE)}\bigg)^{2+\tau} \frac{p^{2+\tau}}{C_2^{1+\tau}},
		\]
		whence
		\begin{equation}
			\label{eq:bound on asymmetric edges1}
		\mu(\CE_{p^k,p^{k+1}}) \le \mu(\CE) \cdot \frac{(MC_2^{1+\tau})^{\frac{1}{2+\tau}}}{p} \le \mu(\CE)\cdot \frac{C_6}{10^{11} p} 
		\end{equation}
		since $C_6\ge C_4=10^{10}M^2C_2^2\ge 10^{12}$ from \eqref{eq:CDefs}. Similarly, we find that
		\begin{equation}
			\label{eq:bound on asymmetric edges2}
		\mu(\CE_{p^k,p^{k-1}}),\ 	\mu(\CE_{p^{k+1},p^k}),\ 	\mu(\CE_{p^{k-1},p^k}) \le  \mu(\CE)\cdot \frac{C_6}{10^{11} p} ,
		\end{equation}
			with the convention that $\CE_{p^0,p^{-1}}=\CE_{p^{-1},p^0}=\emptyset$. 		Combining \eqref{eq:bound on edges from small sets}, \eqref{eq:bound on remote edge sets}, \eqref{eq:bound on asymmetric edges1} and \eqref{eq:bound on asymmetric edges2} completes the proof of the lemma.
	\end{proof}

	\begin{cororollary}\label{cor:primes not dividing both}
		Let $G=(\mu,\CV,\CW,\CE,\CP,f,g)$, $p$ and $G'$ be as in Lemma \ref{lem:bound on edge set for anatomy}. Then one of the following holds:
		\begin{enumerate}[label=(\alph*)]
		\item $\ds
		\mu\Bigl(\Bigl\{(v,w)\in\CE: p|\frac{v w}{\gcd(v,w)^2}\Bigr\}\Bigr) 
		\le \frac{C_6}{10^{10} p} \cdot \mu(\CE)$;
		\item $q(G')> q(G)$.
		\end{enumerate}
	\end{cororollary}
	
	\begin{proof} Let $k\in\Z_{\ge0}$, $\CV'$ and $\CW'$ be as in Lemma \ref{lem:bound on edge set for anatomy}. Assume that conclusion (b) is false. Then conclusion (a) of Lemma \ref{lem:bound on edge set for anatomy} must be true. In addition, note that if $p|v w/\gcd(v,w)^2$, then $p^j\|v$ and $p^\ell\|w$ for some $j\ne \ell$. In particular we cannot have $p^k \|v$ and $p^k\|w$.  Thus
		\[
			\mu\Bigl(\Bigl\{(v,w)\in\CE: p|\frac{v w}{\gcd(v,w)^2}\Bigr\}\Bigr)
			\le \mu\bgg( \CE \bg( (\CV\times\CW')\cup (\CV'\times \CW)\bg)\bgg)  
			\le \frac{C_6}{10^{10} p} \cdot \mu(\CE) ,
		\]
		as claimed.
	\end{proof}

	\section{Proof of the anatomical Lemmas \ref{lem:Cosmetic-L} and \ref{lem:Cosmetic-omega}}\label{sec:proof-anatomy-lemma}

	\begin{proof}[Proof of Lemma \ref{lem:Cosmetic-L}]		
		For brevity, let
		\[
		L_t'(v,w) = \sum_{\substack{p|v w/\gcd(v,w)^2\\p\in \CR(G),\ p>t}}\frac{1}{p}.
		\]
		Since we have assumed that $G$ is maximal, conclusion (b) of Corollary \ref{cor:primes not dividing both} cannot hold, and thus
		\begin{equation}
		\label{eq:not too many assymetric vertices}
		\mu\Bigl(\Bigl\{(v,w)\in\CE: p|\frac{v w}{\gcd(v,w)^2}\Bigr\}\Bigr) 
		\le \frac{C_6}{10^{10}p} \cdot \mu(\CE)
		\end{equation}
		for every prime $p>C_6$ lying in $\CR(G)$. Consequently, 
		\begin{align*}
		\sum_{(v,w)\in\CE } \mu(v)\mu(w) 	L_t'(v,w)
			&=   	\sum_{\substack{p\in \CR(G) \\ p>t}}\frac{1}{p}\cdot 
		\mu\Bigl(\Bigl\{(v,w)\in\CE: p|\frac{v w}{\gcd(v,w)^2}\Bigr\}\Bigr)  \\
			&\le \sum_{p>t} \frac{C_6\mu(\CE)}{10^{10}p^2}<\frac{C_6\mu(\CE)}{10^{10}t}
			\le \frac{\mu(\CE)}{10^{10}s}
		\end{align*}
		by Corollary \ref{cor:primes not dividing both} and our assumption that $t\ge C_6s\ge C_6$  (recall \eqref{eq:CDefs}). Now, let us define
		\[
		\CE':=\{(v,w)\in\CE:  	L_t'(v,w) \le 1/(2s)\}.
		\]
		Evidently, we have that
		\[
		\mu(\CE\setminus\CE')=\sum_{\substack{(v,w)\in\CE \\ 	L_t'(v,w)>1/(2s)}}\mu(v)\mu(w)< 2s \sum_{(v,w)\in\CE } \mu(v)\mu(w) 	L_t'(v,w) < \frac{\mu(\CE)}{1000}.
		\]
		Thus $\mu(\CE')\ge \frac{999}{1000}\mu(\CE)$. We then take $G':=(\mu,\CV,\CW,\CE',\CP,f,g)$ and note that 
		\[
		\frac{q(G')}{q(G)}=\bggg( \frac{\mu(\CE')}{\mu(\CE)}\bggg)^{2+\tau}\ge \frac{1}{2} .
		\]
		Hence, since $\CE'\subseteq\CE\subseteq \{(v,w)\in \CV\times\CW: L_t(v,w)\ge 1/s\}$, for any $(v,w)\in\CE'$ we have
		\[
		\sum_{\substack{p|v w/\gcd(v,w)^2\\p\notin \CR(G),\ p>t}}\frac{1}{p}= L_t(v,w)-L_t'(v,w) 
		\ge \frac{1}{s}-\frac{1}{2s}=\frac{1}{2s}.
		\]
		This completes the proof.
	\end{proof}
	
		\begin{proof}[Proof of Lemma \ref{lem:Cosmetic-omega}]		
		For brevity, let
		\[
		\omega_t'(v,w) = \sum_{\substack{p|v w/\gcd(v,w)^2\\p\in \CR(G),\ C_6<p\le t}}  1. 
		\]
		As before, using the maximality of $G$ yields 	\eqref{eq:not too many assymetric vertices} for each prime $p>C_6$ lying in $\CR(G)$. Hence,
		\begin{align*}
			\sum_{(v,w)\in\CE } \mu(v)\mu(w) 	\omega_t'(v,w)
			&=   	\sum_{\substack{p\in \CR(G) \\ C_6<p\le t}}1\cdot 
			\mu\Bigl(\Bigl\{(v,w)\in\CE: p|\frac{v w}{\gcd(v,w)^2}\Bigr\}\Bigr)  \\
			&\le \sum_{C_6<p\le t} \frac{C_6\mu(\CE)}{10^{10} p} \\
			&\le \Bigl(\log\log{t}-\log\log{C_6}+\frac{1}{2(\log{t})^2}+\frac{1}{2(\log{C_6})^2}\Bigr)\frac{C_6\mu(\CE)}{10^{10}}\\
			&\le \frac{C_6\log\log t}{10^{10}} \cdot \mu(\CE) ,
		\end{align*}
		where we used \cite[Theorem 5]{RS} in the third line and the fact $C_6>10^{10}$ in the final line. Now, let us define
		\[
		\CE':=\Bigl\{(v,w)\in\CE:  	\omega_t'(v,w) \le \frac{C_6}{10^9}\log\log t\Bigr\}.
		\]
		Evidently, we have that
		\[
		\mu(\CE\setminus\CE')=\sum_{\substack{(v,w)\in\CE \\ 	\omega_t'(v,w)>C_610^{-9} \log\log t} }\mu(v)\mu(w)< \sum_{(v,w)\in\CE } \mu(v)\mu(w)\cdot  \frac{\omega_t'(v,w)}{C_6 10^{-9} \log\log t} \le \frac{\mu(\CE)}{10}.
		\]
		Thus $\mu(\CE')\ge 9\mu(\CE)/10$. We then take $G':=(\mu,\CV,\CW,\CE',\CP,f,g)$ and note that 
		\[
		\frac{q(G')}{q(G)}=\Bigl(\frac{\mu(\CE')}{\mu(\CE)}\Bigr)^{2+\tau}\ge \frac{1}{2} .
		\]
		By assumption we have $\CE\subseteq \{(v,w)\in \CV\times\CW: \omega_t(v,w)\ge U\}$. Thus, since $\CE'\subseteq\CE$, for any $(v,w)\in\CE'$ we have
		\[
		\sum_{\substack{p|v w/\gcd(v,w)^2\\p\notin \CR(G),\ C_6<p\le t}}1 \ge \omega_t(v,w)-\omega_t'(v,w) -C_6 
		\ge U-\frac{C_6}{10^9}\log\log t-C_6 \ge \frac{U}{2}
		\]
		since we have assumed that $U\ge 2C_6\log\log t\ge 2C_6^2$. This completes the proof.
	\end{proof}

	\section{Proof of Proposition \ref{prop:IterationStep1}} \label{sec:IterationStep1}
	\begin{lemma}[Bounds on edge sets]\label{lem:EdgeSets}
		Consider a GCD graph $G=(\mu,\CV,\CW,\CE,\CP,f,g)$ with edge density $\delta>0$ and a prime $p\in\CR(G)$. 
		For each $k,\ell\in\Z_{\geqslant0}$, let
		\[
		\alpha_k=\frac{\mu(\CV_{p^k})}{\mu(\CV)}
		\quad\text{and}\quad
		\beta_\ell=\frac{\mu(\CW_{p^\ell})}{\mu(\CW)}.
		\]
		Then there exist $k,\ell\in \Z_{\geqslant 0}$ such that $\alpha_k,\beta_\ell>0$ and
		\[
		\frac{\mu(\mathcal{E}_{p^k,p^\ell})}{\mu(\mathcal{E})}\geqslant \begin{cases}
			(\alpha_k\beta_k)^{(1+\tau)/(2+\tau)},\qquad &\text{if}\ \ k=\ell,\\
			\displaystyle \frac{\alpha_k(1-\beta_k)+\beta_k(1-\alpha_k)+\alpha_\ell(1-\beta_\ell)+\beta_\ell(1-\alpha_\ell)}{2^{|k-\ell|/20} \times C_1}, &\text{if}\ \ k\ne \ell .
		\end{cases}
		\]\\
		
	\end{lemma}
	\begin{proof} Let $\CX=\{(k,\ell)\in\Z_{\geqslant 0}^2: \alpha_k,\beta_\ell>0\}$. Note that if $(k,\ell)\in\Z_{\geqslant0}^2\setminus \CX$, then $\mu(\CE_{p^k,p^\ell})\leqslant \mu(\CV_{p^k})\mu(\CW_{p^\ell})=0$. Thus $\sum_{(k,\ell)\in \CX}\mu(\CE_{p^k,p^\ell}) = \mu(\CE)$. Hence, if we assume that the inequality in the statement of the lemma does not hold for any pair $(k,\ell)\in \CX$, we must have
		\[
		1=\sum_{(k,\ell)\in\CX}\frac{\mu(\CE_{p^k,p^\ell})}{\mu(\CE)}< S_1+S_2,
		\]
		where
		\[
		S_1 \coloneqq \sum_{k=0}^\infty(\alpha_k\beta_k)^{(1+\tau)/(2+\tau)}
		\]
		and
		\[
		S_2\coloneqq \mathop{\sum\sum}_{\substack{k,\ell\ge0\\ k\neq\ell}}
		\frac{\alpha_k(1-\beta_k)
			+\beta_k(1-\alpha_k)
			+\alpha_\ell(1-\beta_\ell)
			+\beta_\ell(1-\alpha_\ell)}{2^{|k-\ell|/20}\times C_1}.
		\]
		Thus, to arrive at a contradiction, it suffices to show that
		\[
		S_1+S_2\leqslant 1.
		\]
		First of all, note that $\sum_{|j|\ge1}2^{-|j|/20}=2/(2^{1/20}-1)\le 100$ and recall from \eqref{eq:CDefs} that $C_1=10^4/\tau$, whence	
		\begin{align*}
			S_2&\leqslant 
			\frac{\tau}{100}
			\bigg(\sum_{k=0}^\infty\alpha_k(1-\beta_k)
			+\sum_{k=0}^\infty\beta_k(1-\alpha_k)
			+\sum_{\ell=0}^\infty\alpha_\ell(1-\beta_\ell)
			+\sum_{\ell=0}^\infty\beta_\ell(1-\alpha_\ell)\bigg) \\
			&=\frac{\tau}{50}\bigg(\sum_{k=0}^\infty\alpha_k(1-\beta_k)
			+ \sum_{\ell=0}^\infty\beta_\ell(1-\alpha_\ell)\bigg).
		\end{align*}
		Observing that 
		\[
		1-\beta_k=\sum_{\ell\ge0,\ \ell\neq k} \beta_\ell
		\quad\text{and}\quad
		1-\alpha_\ell=\sum_{k\ge0,\ k\neq \ell} \alpha_k,
		\]
		we conclude that
		\[
		S_2\leqslant \frac{\tau}{25} \mathop{\sum\sum}_{\substack{k,\ell\geqslant0\\ k\neq\ell}} \alpha_k\beta_\ell
		=  \frac{\tau}{25}\bggg(1-\sum_{k=0}^\infty \alpha_k\beta_k\bggg) .
		\]
		
		Let us now study $S_1$. Since $\alpha_k,\beta_\ell$ are non-negative reals which sum to 1, there exists some $k_0\ge0$ such that 
		\[
		\gamma \coloneqq \max_{k\geqslant0}\alpha_k\beta_k= \alpha_{k_0}\beta_{k_0}.
		\]
		We thus find that 
		\[
		S_1 = \sum_{k=0}^\infty(\alpha_k\beta_k)^{\frac{1+\tau}{2+\tau}}\leqslant \gamma^{\frac{\tau}{2(2+\tau)}}\sum_{k=0}^\infty(\alpha_k\beta_k)^{\frac{1}{2}}\leqslant \gamma^{\frac{\tau}{2(2+\tau)}}\Bigl(\sum_{k=0}^\infty \alpha_k\Bigr)^{\frac{1}{2}}\Bigl(\sum_{\ell=0}^{\infty}\beta_\ell\Bigr)^{\frac{1}{2}}= \gamma^{\frac{\tau}{2(2+\tau)}}
		\]
		where we used the Cauchy--Schwarz inequality to bound $\sum_k (\alpha_k\beta_k)^{1/2}$ from above. 
		We also find that
		\[
		S_2\leqslant \frac{\tau}{25} \bggg(1- \sum_{k=0}^\infty \alpha_k\beta_k\bggg)
		\leqslant \frac{\tau}{25}(1-\gamma). 
		\]
		As a consequence, 
		\[
		S_1+S_2\leqslant \gamma^{\frac{\tau}{2(2+\tau)}}+\frac{\tau}{25}(1-\gamma).
		\]
		The function $x\mapsto x^{\frac{\tau}{2(2+\tau)}}+\frac{\tau}{25}(1-x)$ is increasing for $0\leqslant x\leqslant 1$, and so maximized at $x=1$. Thus we infer that $S_1+S_2\leqslant 1$ as required, completing the proof of the lemma.
	\end{proof}

	\begin{lemma}[Quality increment unless a prime power divides almost all]\label{lem:MainLem}
		Consider a GCD graph $G=(\mu,\CV,\CW,\CE,\CP,f,g)$ with edge density $\delta>0$ and a prime $p\in\CR(G)$ with $p>C_2$. Then one of the following holds:
		\begin{enumerate}
			\item There is a GCD subgraph $G'$ of $G$ with multiplicative data $(\CP',f',g')$ and edge density $\delta'>0$ such that:
			\begin{enumerate}
				\item $G'$ is $(\CP',f',g')$-maximal;
				\item $	\CP'=\CP\cup\{p\}$;
				\item $\CR(G')\subseteq \CR(G)\setminus\{p\}$;
				\item $q(G') \geqslant M^{\un_{f'(p)\neq g'(p)}} \cdot q(G)$;
				\item $p^{f'(p)}\|v$ and $p^{g'(p)}\| w$ for all $(v,w)\in \CV\times \CW$. 
			\end{enumerate}
			\item There is some $k\in\Z_{\geqslant0}$ such that
			\[
			\frac{\mu(\CV_{p^k})}{\mu(\CV)}\geqslant 1-\frac{C_2}{p}
			\quad\text{and}\quad
			\frac{\mu(\CW_{p^k})}{\mu(\CW)}\geqslant 1-\frac{C_2}{p}.
			\]
		\end{enumerate}
	\end{lemma}
	
	\begin{proof} Let $\alpha_k$ and $\beta_\ell$ be defined as in the statement of Lemma \ref{lem:EdgeSets}. Consequently, there are $k,\ell\in\Z_{\geqslant0}$ such that $\alpha_k,\beta_\ell>0$ and
		\begin{equation}\label{eq:FirstPigeonholeArgument}
			\frac{\mu(\CE_{p^k,p^\ell})}{\mu(\CE)} 	
			\geqslant 	
			\begin{cases}
				(\alpha_k\beta_k)^{(1+\tau)/(2+\tau)},\qquad &\text{if}\ \ k=\ell,\\
				\displaystyle \frac{\alpha_k(1-\beta_k)+\beta_k(1-\alpha_k)+\alpha_\ell(1-\beta_\ell)+\beta_\ell(1-\alpha_\ell)}{2^{|k-\ell|/20}\times C_1}, &\text{if}\ \ k\ne \ell
			\end{cases}
		\end{equation}
		In particular, $\mu(\CE_{p^k,p^\ell})>0$, so that $G_{p^k,p^\ell}$ is a non-trivial GCD subgraph of $G$. 
		
		Let $(\CP',f',g')$ be the multiplicative data of $G_{p^k,p^\ell}$, and let $G'$ be a $(\CP',f',g')$-maximal GCD subgraph of $G_{p^k,p^\ell}$. We claim that either $q(G') \geqslant M^{\un_{k\neq \ell}} q(G)$, so that $G'$ satisfies conclusion (a) of the lemma, or that conclusion (b) holds.
		
		We separate two cases, according to whether $k=\ell$ or not.
		
		\bigskip
		
		\noindent 
		\textit{Case 1: $k=\ell$.} 	Lemma \ref{lem:InducedGraphs} and our lower bound $\mu(\CE_{p^k,p^k})/\mu(\CE)\geqslant(\alpha_k\beta_k)^{(1+\tau)/(2+\tau)}$ imply that
		\als{
			\frac{q(G')}{q(G)}	\ge 			\frac{q(G_{p^k,p^\ell}) }{q(G)}=\bigg(\frac{\mu(\CE_{p^k,p^k})}{\mu(\CE)}\bigg)^{2+\tau}\cdot
			\frac{1}{(\alpha_k\beta_k)^{1+\tau}}
			\cdot\frac{1}{(1-\un_{k\geqslant 1}/p)^2(1-1/p^{1+\tau/4})^{3}} \geqslant 1.
		}
		This establishes conclusion  (a) in this case, since $f'(p)=g'(p)=k$ and thus $\un_{f'(p)\ne g'(p)}=0$.

		\bigskip
		
		\noindent 
		\textit{Case 2: $k\ne \ell$}. 	As before, we use Lemma \ref{lem:InducedGraphs} and our lower bound on $\mu(\CE_{p^k,p^\ell})$ to find that
		\begin{align*}
			\frac{q(G')}{q(G)}
			&\ge \bigg(\frac{\mu(\CE_{p^k,p^\ell})}{\mu(\CE)}\bigg)^{2+\tau}
			(\alpha_k\beta_\ell)^{-1-\tau}
			\frac{p^{|k-\ell|}}{(1-1/p^{1+\tau/4})^{3}} \\
			&\geqslant\frac{S^{2+\tau}}{C_1^{2+\tau}(\alpha_k\beta_\ell)^{1+\tau}}
			\cdot \bggg(\frac{p}{2^{\frac{2+\tau}{20}}}\bggg)^{|k-\ell|},
		\end{align*}
		where
		\[
		S=\alpha_k(1-\beta_k)
		+\beta_k(1-\alpha_k)
		+\alpha_\ell(1-\beta_\ell)
		+\beta_\ell(1-\alpha_\ell).
		\]
		Note that 
		\begin{equation}
			\label{S lb1}
			S\geqslant \alpha_k(1-\beta_k)\geqslant \alpha_k\beta_\ell.
		\end{equation}
		Indeed, this follows by our assumption that $k\neq\ell$, which implies that $\beta_k+\beta_\ell\leqslant \sum_{j\ge 0}\beta_j=1$. Combining the above, we conclude that
		\begin{equation}
		\label{q'/q-lb} 
			 \frac{q(G')}{q(G)}
			\geqslant \frac{S^{2}}{\alpha_k\beta_\ell}
			\cdot \frac{1}{C_1^{2+\tau}} \cdot \bggg(\frac{p}{2^{\frac{2+\tau}{20}}}\bggg)^{|k-\ell|} .
		\end{equation}
		If $q(G')\ge M\cdot q(G)$, we are done. So let us assume that $q(G')<M\cdot q(G)$. Since $|k-\ell|\geqslant 1$, we must then have that
		\[
		S\leqslant \frac{S^{2}}{\alpha_k\beta_\ell}
		\leqslant M\cdot C_1^{2+\tau} \cdot  \bggg(\frac{2}{p}\bggg)^{|k-\ell|}
		\leqslant \frac{2MC_1^{2.01}}{p}\le \frac{C_2}{10p}
		\leqslant \frac{1}{10},
		\]
		where we used our assumption that $p\geqslant C_2 = 10MC_1^3$ for the second-to-last inequality  (recall \eqref{eq:CDefs}). In particular, this gives
		\begin{equation}\label{small quality}
			S\leqslant  \frac{C_2}{10p}
			\quad\text{and}\quad
			\frac{S^2}{\alpha_k\beta_\ell}\leqslant \frac{1}{10}.
		\end{equation}
		
		We note that
		\begin{equation}
			S\geqslant\alpha_k(1-\beta_k)+\beta_\ell(1-\alpha_\ell)\geqslant (\alpha_k+\beta_\ell)(1-\max\{\alpha_\ell,\beta_k\}).
			\label{eq:SBound}
		\end{equation}
		Thus by the arithmetic-geometric mean inequality, and relations \eqref{eq:SBound} and \eqref{small quality}, we have
		\[
		(1-\max\{\alpha_\ell,\beta_k\})^2\leqslant  \frac{(\alpha_k+\beta_\ell)^2}{4\alpha_k\beta_\ell}(1-\max\{\alpha_\ell,\beta_k\})^2\leqslant \frac{S^2}{4\alpha_k\beta_\ell}\leqslant \frac{1}{40}.
		\]
		In particular, $\max\{\alpha_\ell,\beta_k\}\geqslant 1/2$.
		
		We consider the case when $\beta_k\geqslant1/2$; the case with $\alpha_\ell\geqslant 1/2$ is entirely analogous with the roles of $\beta$ and $\alpha$ swapped, and the roles of $k$ and $\ell$ swapped. Thus, to complete the proof of the lemma, it suffices to show that
		\begin{equation}\label{small quality - large structure}
			\alpha_k,\beta_k
			\geqslant 1-\frac{C_2}{p}.
		\end{equation}
		The first inequality of \eqref{small quality} states that
		\[
		\alpha_k(1-\beta_k)
		+\beta_k(1-\alpha_k)
		+\alpha_\ell(1-\beta_\ell)
		+\beta_\ell(1-\alpha_\ell)
		\leqslant  \frac{C_2}{10p}.
		\]
		Since $\beta_k\geqslant1/2$ and $p\ge C_2$, we infer that
		\[
		1-\alpha_k\leqslant 2\beta_k(1-\alpha_k)\leqslant 
		\frac{2 C_2}{10p}\leqslant \frac{1}{5}. 
		\]
		In particular, $\alpha_k\geqslant 1-C_2/p$ and $\alpha_k\geqslant1/2$, whence
		\[
		1-\beta_k\leqslant 2\alpha_k(1-\beta_k)\leqslant 
		\frac{2C_2}{10p}\leqslant \frac{C_2}{p}.
		\]
		This completes the proof of \eqref{small quality - large structure} and hence of the lemma.
	\end{proof}

	\begin{proof}[Proof of Proposition \ref{prop:IterationStep1}]
		This follows almost immediately from Lemma \ref{lem:MainLem}. Since $\CR(G)\subseteq\{p> C_6\}$ by assumption, if $p\in\CR(G)$ then $p> C_6$. We have also assumed that $\CR^{\flat}(G)\ne\emptyset$. Consequently, there is a prime $p\in\CR^\flat(G)$ with $p> C_6> C_2$. We then have two cases:
		
		\medskip
		
		\noindent {\it Case 1:} there exists an integer $k\ge0$ such that 
		\[	
		\frac{\mu(\CV_{p^k})}{\mu(\CV)}\geqslant 1-\frac{C_2}{p}
		\quad\text{and}\quad
		\frac{\mu(\CW_{p^k})}{\mu(\CW)}\geqslant 1-\frac{C_2}{p}.
		\]
		Since $p\notin \CR^\sharp(G)$, there must exist some a pair of distinct integers $a,b\ge0$ such that $q(G_{p^a,p^b})\ge Mq(G)$. We then take $G'$ to be a $(\CP\cup\{p\},f_{p^a},g_{p^b})$-maximal GCD subgraph of $G_{p^a,p^b}$. To complete the proof of the proposition, note that the condition $\delta>0$ implies that $q(G)>0$. Therefore $q(G')\ge q(G_{p^a,p^b})\ge Mq(G)>0$, which also means that $\delta'>0$.
		
		\medskip
		
		\noindent {\it Case 2:} for every $k\in\Z_{\ge0}$, we have
		\begin{equation}
		\label{eq:Rflat - main case}
		\min\bggg\{ \frac{\mu(\CV_{p^k})}{\mu(\CV)}, \frac{\mu(\CW_{p^k})}{\mu(\CW)} \bggg\} < 1-\frac{C_2}{p}.
				\end{equation}
		 We may then apply Lemma \ref{lem:MainLem} with this choice of $p$. By \eqref{eq:Rflat - main case},  conclusion $(b)$ cannot hold, and so conclusion $(a)$ must hold. This then gives the result.
	\end{proof}

	\section{Proof of Proposition \ref{prop:SmallPrimes}}\label{sec:SmallPrime}

	In this section we prove Proposition \ref{prop:SmallPrimes}.

	\begin{lemma}[Small quality loss or prime power divides positive proportion]\label{lem:SmallPrime}
		Consider a GCD graph $G=(\mu,\CV,\CW,\CE,\CP,f,g)$ with edge density $\delta>0$, and let $p\in\CR(G)$ be a prime.  Then one of the following holds:
		\begin{enumerate}[label=(\alph*)]
			\item There is a GCD subgraph $G'$ of $G$ with multiplicative data $(\CP',f',g')$ and edge density $\delta'>0$ such that
			\[
			\CP'=\CP\cup\{p\},\quad
			\CR(G')\subseteq\CR(G)\setminus\{p\},\quad 
		q(G')\ge q(G)/C_3 .
			\]
			\item There is some $k\in\Z_{\geqslant0}$ such that
			\[
			\frac{\mu(\CV_{p^k})}{\mu(\CV)}\geqslant \frac{9}{10}
			\quad\text{and}\quad
			\frac{\mu(\CW_{p^k})}{\mu(\CW)}\geqslant \frac{9}{10}.
			\]
		\end{enumerate}
	\end{lemma}
	\begin{proof}
		Assume that conclusion $(a)$ does not hold, so we intend to establish $(b)$. For $k,\ell\in\Z_{\ge0}$, let $\mu(\CV_{p^k})=\alpha_k\mu(\CV)$ and $\mu(\CW_{p^\ell})=\beta_\ell\mu(\CW)$. We begin as in the proof of  Lemma \ref{lem:MainLem}, by considering $k,\ell\in\Z_{\ge0}$ satisfying \eqref{eq:FirstPigeonholeArgument} and the inequalities $\alpha_k,\beta_\ell>0$. We do not need to worry about the maximality of $G'$ here, so we shall simply take $G'=G_{p^k,p^\ell}$. In particular, $G'$ is a non-trivial GCD subgraph of $G$.
		
		We note that the proof of Lemma \ref{lem:MainLem} up to relation \eqref{q'/q-lb} requires no assumption on the size of $p$. Now, if $k=\ell$, then Case 1 of the proof of Lemma \ref{lem:MainLem} shows that conclusion $(a)$ must hold, contradicting our assumption. Therefore we may assume that $k\ne \ell$. Now, arguing as in Case 2 of the proof of Lemma \ref{lem:MainLem}, and setting
		\[
		S=\alpha_k(1-\beta_k)
		+\beta_k(1-\alpha_k)
		+\alpha_\ell(1-\beta_\ell)
		+\beta_\ell(1-\alpha_\ell),
		\]
		we infer that
		\[
		\frac{1}{C_3}\geqslant \frac{q(G')}{q(G)}
		\geqslant\frac{S^{2}}{\alpha_k\beta_\ell} \cdot \frac{1}{C_1^{2+\tau}}
		\cdot \bggg(\frac{p}{2^{  \frac{2+\tau}{20}  }}\bggg)^{|k-\ell|}\geqslant \frac{S^{2}}{\alpha_k\beta_\ell} \cdot \frac{1}{C_1^3}.
		\]
		Therefore we have that
		\[
		S\leqslant \frac{S^2}{\alpha_k\beta_\ell}\leqslant \frac{C_1^3}{C_3} \leqslant \frac{1}{10^3}.
		\]
		Since $S\geqslant (\alpha_k+\beta_\ell)(1-\max\{\alpha_\ell,\beta_k\})$, we have
		\[
		(1-\max\{\alpha_\ell,\beta_k\})^2\leqslant \frac{ (\alpha_k+\beta_\ell)^2}{4\alpha_k\beta_\ell}(1-\max\{\alpha_\ell,\beta_k\})^2\leqslant \frac{S^2}{4\alpha_k\beta_\ell}\leqslant \frac{1}{100},
		\]
		so $\max\{\alpha_\ell,\beta_k\}\geqslant 9/10$. We deal with the case when $\beta_k\geqslant 9/10$; the case with $\alpha_\ell\geqslant 9/10$ is entirely analogous with the roles of $k$ and $\ell$ and the roles of $\alpha$ and $\beta$ swapped. 
		
		Since $\beta_k\geqslant 9/10$, we have
		\[
		1-\alpha_k\leqslant 2\beta_k(1-\alpha_k)\leqslant 2S\leqslant \frac{2}{10^{3}}
		\]
		In particular, $\alpha_k \geqslant 9/10$ and so conclusion $(b)$ holds, as required.
	\end{proof}

	\begin{lemma}[Adding small primes to $\CP$]\label{lem:SmallIteration}
		Let $G=(\mu,\CV,\CW,\CE,\CP,f,g)$ be a GCD graph with edge density $\delta>0$. Let $p\in\CR(G)$ be a prime.
		
		Then there is a GCD subgraph $G'$ of $G$ with set of primes $\CP'$ and edge density $\delta'>0$ such that 
		\[
		\CP'=\CP\cup\{p\},
		\quad \CR(G')\subseteq\CR(G)\setminus\{p\},
		\quad 
		q(G')\ge q(G)/C_5.
		\]
	\end{lemma}
	\begin{proof}
		We first repeatedly apply Lemma \ref{lem:HighDegree} until we arrive at a GCD subgraph 
		\[
		G^{(1)}=(\mu,\CV^{(1)},\CW^{(1)},\CE^{(1)},\CP,f,g)
		\]
		of $G$ with edge density $\delta^{(1)}$ such that 
		\[
		q(G^{(1)})\geqslant q(G),
		\]
		as well as 
		\[
		\mu(\Gamma_{G^{(1)}}(v))\geqslant  \frac{1+\tau}{2+\tau}\cdot \delta^{(1)}\mu(\CW^{(1)}) 
		\qquad\text{for all}\  v\in \CV^{(1)}.
		\] 
		(We must eventually arrive at such a subgraph since the vertex sets are strictly decreasing at each stage but can never become empty since the edge density remains bounded away from 0.)
		
		We now apply Lemma \ref{lem:SmallPrime} to $G^{(1)}$. If conclusion $(a)$ of Lemma \ref{lem:SmallPrime} holds, then there is a GCD subgraph $G^{(2)}$ of $G^{(1)}$ satisfying the conclusion of Lemma \ref{lem:SmallIteration}, so we are done by taking $G'=G^{(2)}$. Therefore we may assume that instead conclusion $(b)$ of Lemma \ref{lem:SmallPrime} holds, so there is some $k\in\mathbb{Z}_{\ge 0}$ such that 
		\begin{equation}\label{eq:FirstLowerBound}
			\frac{\mu(\CV^{(1)}_{p^k})}{\mu(\CV^{(1)})}\geqslant \frac{9}{10}
			\quad\text{and}\quad
			\frac{\mu(\CW^{(1)}_{p^k})}{\mu(\CW^{(1)})}\geqslant \frac{9}{10}.
		\end{equation}
		In fact we claim that either the conclusion of Lemma \ref{lem:SmallIteration} holds, or we have the stronger condition
		\begin{equation}\label{eq:ImprovedLowerBound}
			\frac{\mu(\CV^{(1)}_{p^k})}{\mu(\CV^{(1)})}\geqslant \max\Bigl(\frac{9}{10},1-\frac{C_2}{p}\Bigr)
			\quad\text{and}\quad
			\frac{\mu(\CW^{(1)}_{p^k})}{\mu(\CW^{(1)})}\geqslant \max\Bigl(\frac{9}{10},1-\frac{C_2}{p}\Bigr).
		\end{equation}
		Relation \eqref{eq:ImprovedLowerBound} follows immediately from \eqref{eq:FirstLowerBound} if $p\leqslant 10 C_2$, so let us assume that $p > 10 C_2$. We then apply Lemma \ref{lem:MainLem} to $G^{(1)}$. If conclusion $(a)$ of Lemma \ref{lem:MainLem} holds, then there is a GCD subgraph $G^{(3)}$ of $G^{(1)}$ satisfying the required conditions of Lemma \ref{lem:SmallIteration}, so we are done by taking $G'=G^{(3)}$. Therefore we may assume that conclusion $(b)$ of Lemma \ref{lem:MainLem} holds, so that there is some $k'\geqslant0$ such that $\mu(\CV^{(1)}_{p^{k'}})/\mu(\CV^{(1)})\geqslant 1-C_2/p\geqslant 9/10$ and $\mu(\CW^{(1)}_{p^{k'}})/\mu(\CW^{(1)}) \geqslant 1-C_2/p\geqslant 9/10$. Since there cannot be two disjoint subsets of $\CV^{(1)}$ of density $\geqslant 9/10$, we must then have $k'=k$, thus proving \eqref{eq:ImprovedLowerBound} in this case too. 
		
		In conclusion, regardless of the size of $p$ we have established \eqref{eq:ImprovedLowerBound}. Next, we fix an integer $r\leqslant r_0 $ such that $p^r > C_4$, where $r_0$ is the smallest integer such that $2^{r_0} > C_4$, namely $r_0 = \lfloor  \frac{\log C_4}{\log 2} \rfloor + 1 $ and we apply Lemma \ref{lem:UnbalancedSetEdges1}.
		
		If conclusion $(a)$ of Lemma \ref{lem:UnbalancedSetEdges1} holds, then we take $G'=G^{(1)}_{p^k,p^\ell}$, whose quality satisfies $q(G')\geqslant M q(G^{(1)})\geqslant Mq(G)>0$. 
		In particular, $\delta'>0$, so the proof is complete in this case.
		
		Thus we may assume that conclusion $(b)$ of Lemma \ref{lem:UnbalancedSetEdges1} holds, so that
		\[
		\sum_{\ell\in\CL_{k,r}}  \mu(\CE^{(1)}_{p^k,p^\ell})\leqslant\frac{\mu(\CE^{(1)})}{  4p^{1+\tau/4}}<\frac{\mu(\CE^{(1)})}{4},
		\]
		where we recall the notation $\CL_{k,r}=\{\ell\in\Z_{\geqslant0}:|\ell-k|\geqslant r+1\}$. Let 
		\[
		\tilde{\CW}^{(1)}=\bigcup_{\substack{\ell\geqslant0\,:\, |\ell-k|\leqslant r}}\CW^{(1)}_{p^\ell}
		\]
		and let
		\[
		\CE^{(2)}=\CE^{(1)}\cap(\CV^{(1)}_{p^k}\times\tilde{\CW}^{(1)})\subseteq\CE^{(1)}
		\] 
		be the set of edges between $\CV^{(1)}_{p^k}$ and $\tilde{\CW}^{(1)}$ in $G^{(1)}$. Since $\mu(\CV^{(1)}_{p^k})\geqslant 9\mu(\CV^{(1)})/10$ and $\mu(\Gamma_{G^{(1)}}(v))\geqslant \frac{1+\tau}{2+\tau}\cdot  \delta^{(1)}\mu(\CW^{(1)})\geqslant \frac{1}{2}\cdot  \delta^{(1)}\mu(\CW^{(1)})$ for all $v\in \CV^{(1)}_{p^k}$, we have
		\begin{align*}
			\mu(\CE^{(2)})\geqslant \mu\bg(\CE^{(1)}\cap(\CV_{p^k}^{(1)}\times\CW^{(1)}) \bg) -\sum_{\ell\in\CL_{k,r}}  \mu(\CE^{(1)}_{p^k,p^\ell})
			&\geqslant \sum_{v\in \CV^{(1)}_{p^k}}\mu(v)\mu(\Gamma_{G^{(1)}}(v))-\frac{\mu(\CE^{(1)})}{4}\\
			&\geqslant \frac{1}{2}\cdot \delta^{(1)}\mu(\CV^{(1)}_{p^k})\mu(\CW^{(1)})-\frac{\mu(\CE^{(1)})}{4}\\
			&\geqslant \left(\frac{1}{2}\cdot\frac{9}{10}-\frac{1}{4}\right)\cdot \mu(\CE^{(1)}) = \frac{\mu(\CE^{(1)})}{5} >0. 
		\end{align*}
		Let $G^{(2)}=(\mu,\CV^{(1)}_{p^k},\tilde{\CW}^{(1)},\CE^{(2)},\CP,f,g)$ be the GCD subgraph of $G^{(1)}$ formed by restricting to $\CV^{(1)}_{p^k}$ and $\tilde{\CW}^{(1)}$. Since $\mu(\CE^{(2)})>0$, $G^{(2)}$ is a non-trivial GCD subgraph. In addition, we have that
		\begin{align*}
			\frac{q(G^{(2)})}{q(G^{(1)})}
			=\bigg(\frac{\mu(\CE^{(2)})}{\mu(\CE^{(1)})}\bigg)^{2+\tau}
			\bigg(\frac{\mu(\CV^{(1)})}{\mu(\CV_{p^k}^{(1)})}\bigg)^{1+\tau}
			\bigg( \frac{\mu(\CW^{(1)})}{\mu(\tilde{\CW}^{(1)})}\bigg)^{1+\tau}
			\geqslant\bigg(\frac{1}{5}\bigg)^{2+\tau}\cdot 1\cdot 1 \ge \frac{1}{5^{3}}.
		\end{align*}
		Finally, we apply Lemma \ref{lem:Pigeonhole} to the partition 
		\[
		\tilde{\CW}^{(1)}=\bigsqcup_{\ell\ge0\,:\, |\ell-k|\leqslant r}\CW^{(1)}_{p^\ell}
		\]
		of $\tilde{\CW}^{(1)}$ into $\leqslant 2\cdot r +1\leqslant 10 \log C_4$ subsets. This produces a GCD subgraph
		\[
		G^{(3)}=(\mu,\CV^{(1)}_{p^k},\CW^{(1)}_{p^\ell},\CE^{(1)}_{p^k,p^\ell},\CP,f,g)
		\]
		of $G^{(2)}$ for some $\ell\geqslant 0$ with $|\ell-k|\leqslant r$ such that
		\[
		q(G^{(3)})\geqslant \frac{q(G^{(2)})}{\bg(10 \log C_4\bg)^{2+\tau} }\geqslant \frac{q(G^{(1)})}{\bg(10 \log C_4 \bg)^3\cdot 5^{3}}\geqslant \frac{q(G)}{\bg(50 \log C_4\bg)^3}.
		\]
		Finally, we note that $G^{(1)}_{p^k,p^\ell}$ is a GCD subgraph of $G^{(3)}$ with set of primes $\CP\cup\{p\}$, edge density $\delta^{(1)}_{p^k,p^\ell}=\delta^{(3)}$, and quality $q(G^{(1)}_{p^k,p^\ell})\geqslant q(G^{(3)})$. Taking $G'=G^{(1)}_{p^k,p^\ell}$ then gives the result on recalling the definition \eqref{eq:CDefs} of $C_5$.
	\end{proof}

	\begin{proof}[Proof of Proposition \ref{prop:SmallPrimes}]
		If $\CR(G)\cap\{p\leqslant C_6\}=\emptyset$, then we can simply take $G'=G$.
		
		If $\CR(G)\cap\{p\leqslant C_6\}\ne\emptyset$, then we can choose a prime $p\in\CR(G)\cap\{p\leqslant C_6\}$ and apply Lemma \ref{lem:SmallIteration}. We do this repeatedly to produce a sequence of GCD subgraphs 
		\[
		G=:G_1\succeq G_2\succeq\cdots
		\]
		such that
		\begin{equation}\label{eq:SmallPrimesInduction}
		q(G_{i+1}) \ge q(G_i)/ C_5 .
		\end{equation}
		for each $i$. In addition, we let $\CP_i$ denote the set of primes associated to $G_i$, so that 
		$\mathcal{P} =\CP_1\subseteq\CP_2\subseteq\cdots \subseteq \mathcal{P} \cup 
		\bg(\mathcal{R}(G) \cap \{ p\leqslant C_6 \}\bg)$. 
		
		At each stage, the set $\CR(G_i)\cap\{p\leqslant C_6\}$ is strictly smaller than before. So, after at most $C_6$ steps we arrive at a GCD subgraph $G^{(1)}=(\mu,\CV^{(1)},\CW^{(1)},\CE^{(1)},\CP^{(1)},f^{(1)},g^{(1)})$ of $G$ with 
		\[
		\CP^{(1)}\subseteq\mathcal{P} \cup \bg(\mathcal{R}(G) \cap \{ p\leqslant C_6 \}\bg)
		\quad\text{and}\quad \CR(G^{(1)})\cap\{p\leqslant C_6 \}=\emptyset.
		\]
		Iterating \eqref{eq:SmallPrimesInduction} at most $C_6$ times, we find that $q(G^{(1)}) \ge q(G)/ C_5^{C_6}=q(G)/C_7$. 	Thus, taking $G'=G^{(1)}$ gives the result.
	\end{proof}
	
	\section{Proof of Proposition \ref{prop:IterationStep2}}\label{sec:IterationStep2}
	
	In this section we prove Proposition \ref{prop:IterationStep2}.

	\begin{lemma}[Quality increment even when a prime power divides almost all]\label{lem:MainLem2} 
		Consider a GCD graph $G=(\mu,\CV,\CW,\CE,\CP,f,g)$ with edge density $\delta>0$ and let $p\in\CR(G)$ be a prime with $p\geqslant C_6$. Then there is a GCD subgraph $G'$ of $G$ with set of primes $\CP'=\CP\cup\{p\}$ such that
		\[
		\CR(G')\subseteq\CR(G)\setminus\{p\}
		\quad\text{and}\quad  q(G')\geqslant q(G)>0.
		\]
	\end{lemma}
	
	\begin{proof} First of all, we may assume without loss of generality that for all sets $\CA\subseteq \CV$ and $\CB\subseteq \CW$, we have that 
		\begin{equation}\label{eq:NoEdgesBetweenSmallSets}
			\mu(\CE(\CA,\CB)) \leqslant \frac{\mu(\CE)}{p^{1+\frac{\tau}{3}}} 
			\quad\text{whenever}\quad \max\bigg\{ \frac{\mu(\CA)}{\mu(\CV)}, \frac{\mu(\CB)}{\mu(\CW)} \bigg\} \leqslant \frac{C_2}{p} .
		\end{equation}
		Indeed, if $G$ does not satisfy \eqref{eq:NoEdgesBetweenSmallSets}, then we apply Lemma \ref{lem:NoSmallSetEdges} with $\eta=C_2/p$ to replace $G$ by a non-trivial subgraph $G^{(1)}$ that does have this property (noticing that $(C_2/p)^{(2+2\tau)/(2+\tau)}\leqslant 1/(p^{1+ \tau/3})$ for $p\geqslant C_6$ from the definition \eqref{eq:CDefs}). In addition, $G^{(1)}$ has the same multiplicative data as $G$ and its quality is strictly larger. Hence, we may work with $G^{(1)}$ instead. So, from now on, we assume that \eqref{eq:NoEdgesBetweenSmallSets} holds. 
		
		We now apply Lemma \ref{lem:MainLem}. If conclusion $(a)$ of Lemma \ref{lem:MainLem} holds, then we are done. Thus we may assume that conclusion $(b)$ holds, that is to say there is some $k\in\Z_{\geqslant0}$ such that 
		\[
		\frac{\mu(\CV_{p^k})}{\mu(\CV)}\geqslant1-\frac{C_2}{p}
		\quad\text{and}\quad
		\frac{\mu(\CW_{p^k})}{\mu(\CW)}\geqslant 1-\frac{C_2}{p}.
		\]
		In particular, by \eqref{eq:NoEdgesBetweenSmallSets} we see that
		\begin{equation}\label{eq:NoBadEdges1}
			\mu\big(\CE(\CV\setminus\CV_{p^k},\CW\setminus \CW_{p^k})\big) 
			\leqslant \frac{\mu(\CE)}{p^{1+\frac{\tau}{3}}}.
		\end{equation}
		
		Now, set
		\[
		\tilde{\CV}_{p^k}=\CV_{p^{k-1}}\cup\CV_{p^k}\cup\CV_{p^{k+1}}
		\quad\text{and}\quad
		\tilde{\CW}_{p^k}=\CW_{p^{k-1}}\cup\CW_{p^k}\cup\CW_{p^{k+1}} ,
		\]
		with the convention that $\CV_{p^{-1}}=\emptyset=\CW_{p^{-1}}$. 
		In view of Lemmas \ref{lem:UnbalancedSetEdges1} and \ref{lem:UnbalancedSetEdges2} applied with $r=1$, we may assume that 
		\begin{equation}\label{eq:NoBadEdges2}
			\mu\big(\CE(\CV\setminus\tilde{\CV}_{p^k},\CW_{p^k})\big)
			= \sum_{\substack{i\ge0 \\ |i-k|\geqslant2}} \mu(\CE(\CV_{p^i},\CW_{p^k})) \leqslant \frac{\mu(\CE)}{4p^{1+\frac{\tau}{4}}},
		\end{equation}
		and
		\begin{equation}\label{eq:NoBadEdges3}
			\mu\big(\CE(\CV_{p^k},\CW\setminus\tilde{\CW}_{p^k})\big)
			= \sum_{\substack{j\ge0 \\ |j-k|\geqslant2}} \mu(\CE(\CV_{p^k},\CW_{p^j})) \leqslant \frac{\mu(\CE)}{4p^{1+\frac{\tau}{4}}}.
		\end{equation}
		Hence, if we let
		\[
		\CE^*=\CE(\CV_{p^k},\tilde{\CW}_{p^k})\cup \CE(\tilde{\CV}_{p^k},\CW_{p^k}) ,
		\]
		then \eqref{eq:NoBadEdges1}-\eqref{eq:NoBadEdges3} imply that
		\[
		\frac{\mu(\CE^*)}{\mu(\CE)} \geqslant 1- \frac{1}{2p^{1+\frac{\tau}{4}}} - \frac{1}{p^{1+\frac{\tau}{3}}} 
		\geqslant \bigg(1-\frac{1}{p^{1+\frac{\tau}{4}}}\bigg)^{\frac{2}{3}  } >0 ,
		\]
		where we used our assumption that $p\geqslant C_6$ and the inequality $(1-x)^{2/3}\leqslant 1-2x/3$ for $x\in[0,1]$ that follows from Taylor's theorem. We then consider the non-trivial GCD subgraph $G^*=(\mu,\CV,\CW,\CE^*,\CP,f,g)$ of $G$ formed by restricting the edge set to $\CE^*$. Note that
		\begin{equation}
			\label{G^*/G}
			\frac{q(G^*)}{q(G)}=\Bigl(\frac{\mu(\CE^*)}{\mu(\CE)}\Bigr)^{2+\tau} \geqslant \bggg(1-\frac{1}{p^{1+\frac{\tau}{4}}}\bggg)^{\frac{4 + 2\tau}{3}} \ge\bggg(1-\frac{1}{p^{1+\frac{\tau}{4}}}\bggg)^2.
		\end{equation}

		We  set $G^+=(\mu,\CV^+,\CW^+,\CE^+,\CP\cup\{p\},f^+,g^+)$, where:
		\[
		\CV^+=\CV_{p^k}\cup\CV_{p^{k+1}},
		\quad
		\CW^+=\CW_{p^k}\cup\CW_{p^{k+1}},
		\quad
		\CE^+=\CE^*\cap(\CV^+\times \CW^+),
		\]
		as well as
		\[
		f^+\big|_\CP=f,
		\quad
		f^+(p)=k,
		\quad
		g^+\big|_{\CP}=g,
		\quad
		g^+(p)=k .
		\]
		It is easy to check that $G^+$ is a GCD subgraph of $G^*$ (and hence of $G$). Note that $\mu(\CV^+)\geqslant \mu(\CV_{p^k})\geqslant 1-C_2/p>0$. Similarly, we have $\mu(\CW^+)>0$. Consequently, its quality satisfies the relation
		\[
		\frac{q(G^+)}{q(G^*)} =
		\bigg(\frac{\mu(\CE^+)}{\mu(\CE^*)}\bigg)^{2+\tau} 
		\bigg( \frac{\mu(\CV)}{\mu(\CV^+)}\bigg)^{1+\tau}
		\bigg( \frac{\mu(\CW)}{\mu(\CW^+)}\bigg)^{1+\tau}
		\bggg(1-\frac{\un_{k\geqslant1}}{p}\bggg)^{-2}
			 \bggg(1- \frac{1}{p^{1+\frac{\tau}{4}}} \bggg)^{-3} .	
		\]
		(This relation is valid even if $\mu(\CE^+)=0$.) We separate two cases.
		
		\bigskip
		
		\noindent 
		{\it Case 1:} $k=0$. 
		
		\medskip
		
		In this case $\CV_{p^{k-1}}=\CW_{p^{k-1}}=\emptyset$, $ \CE^+ = \CE^*$. And note that $\frac{\mu(\CV)}{\mu(\CV^+)}, \frac{\mu(\CW)}{\mu(\CW^+)} \geqslant 1$. As a consequence, 
		\[
		\frac{q(G^+)}{q(G^*)} \geqslant 	 \bggg(1- \frac{1}{p^{1+\frac{\tau}{4}}} \bggg)^{-3} .	
		\]
	Together with \eqref{G^*/G} this implies that
		\[
		\frac{q(G^+)}{q(G)}
		\ge \bggg(1-\frac{1}{p^{1+\frac{\tau}{4}}}\bggg)^{-1} \ge 1.
		\]
		In particular, $q(G^+)>0$. Thus the lemma follows by taking $G'=G^+$.

		\bigskip
		
		\noindent 
		{\it Case 2:} $k\geqslant 1$.

		\medskip
		
		In this case we have
		\begin{align}
							\frac{q(G^+)}{q(G^*)} =
				\bigg(\frac{\mu(\CE^+)}{\mu(\CE^*)}\bigg)^{2+\tau} 
				\bigg( \frac{\mu(\CV)}{\mu(\CV^+)}\bigg)^{1+\tau}
				\bigg( \frac{\mu(\CW)}{\mu(\CW^+)}\bigg)^{1+\tau}
			 \bggg(1-\frac{1}{p}\bggg)^{-2}
			 \bggg(1- \frac{1}{p^{1+\frac{\tau}{4}}} \bggg)^{-3} .	
				\label{eq:G+}
		\end{align}
		We also consider the GCD subgraphs  $G_{p^k,p^{k-1}}$ and $G_{p^{k-1},p^k}$ of $G$. Notice that $\mu(\CV_{p^k})\geqslant1-C_2/p>0$ for $p\geqslant C_6$. Hence, if $\mu(\CW_{p^{k-1}})>0$, then Lemma \ref{lem:InducedGraphs} implies that
		\begin{align}
							\frac{q(G_{p^k,p^{k-1}})}{q(G^*)} =
				\bigg(\frac{\mu(\CE_{p^k,p^{k-1}})}{\mu(\CE^*)}\bigg)^{2+\tau} 
				\bigg( \frac{\mu(\CV)}{\mu(\CV_{p^k})}\bigg)^{1+\tau}
				\bigg( \frac{\mu(\CW)}{\mu(\CW_{p^{k-1}})}\bigg)^{1+\tau}
				\frac{p}{\bg(1- \frac{1}{p^{1+\frac{\tau}{4}}} \bg)^3} .
				\label{eq:G2}
		\end{align}
		Similarly, if $\mu(\CV_{p^{k-1}})>0$, then we have 
		\begin{align}
						\frac{q(G_{p^{k-1},p^k})}{q(G^*)} =
			\bigg(\frac{\mu(\CE_{p^{k-1},p^k})}{\mu(\CE^*)}\bigg)^{2+\tau} 
			\bigg( \frac{\mu(\CV)}{\mu(\CV_{p^{k-1}})}\bigg)^{1+\tau}
			\bigg( \frac{\mu(\CW)}{\mu(\CW_{p^k})}\bigg)^{1+\tau}
			\frac{p}{\bg(1- \frac{1}{p^{1+\frac{\tau}{4} }} \bg)^3} .
			\label{eq:G3}
		\end{align}
		
		Since $\mu(\CV_{p^k})\geqslant(1-C_2/p)\mu(\CV)$, we have that $\mu(\CV_{p^{k-1}})\leqslant C_2\mu(\CV)/p$. Similarly, we have that $\mu(\CW_{p^{k-1}})\leqslant C_2\mu(\CW)/p$. To this end, let $0\leqslant A,B\leqslant C_2$ be such that
		\al{
			\frac{\mu(\CV_{p^{k-1}})}{\mu(\CV)} = \frac{A}{p} 
			\quad\text{and}\quad
			\frac{\mu(\CW_{p^{k-1}})}{\mu(\CW)} = \frac{B}{p}.
			\label{eq:AB}
		}
		We note that this implies that
		\[
		\frac{\mu(\CV^+)}{\mu(\CV)}\leqslant 1-\frac{A}{p}\quad\text{and}\quad\frac{\mu(\CW^+)}{\mu(\CW)} \leqslant  1-\frac{B}{p}.
		\]
		We also note that $\mu(\CE_{p^k,p^{k-1}})\leqslant \mu(\CV_{p^k})\mu(\CW_{p^{k-1}})\leqslant B\mu(\CV)\mu(\CW)/p$, so if $\mu(\CE_{p^k,p^{k-1}})>0$ then $B>0$. Similarly if $\mu(\CE_{p^{k-1},p^{k}})>0$ then $A>0$.

		Combining \eqref{eq:G+} and \eqref{eq:AB} with \eqref{G^*/G}, we find
		\begin{align}		
				\frac{q(G^+)}{q(G)} &\geqslant
			\bigg(\frac{\mu(\CE^+)}{\mu(\CE^*)}\bigg)^{2+\tau} 
			\frac{1}{\bg(1- \frac{A}{p}\bg)^{1+\tau} \bg(1- \frac{B}{p}\bg)^{1+\tau} \bg(1-  \frac{1}{p}\bg)^2 \bg(1-  \frac{1}{p^{1+\frac{\tau}{4} } }\bg)} .
			\label{eq:ineqA}
		\end{align}
		Similarly, provided $B>0$, \eqref{eq:G2}, \eqref{eq:AB} and \eqref{G^*/G} give
		\begin{align}
		\frac{q(G_{p^k,p^{k-1}})}{q(G)}
			& \geqslant
			\bigg(\frac{\mu(\CE_{p^k,p^{k-1}})}{\mu(\CE^*)}\bigg)^{2+\tau} 
			\frac{p^{2+\tau}}{B^{1+\tau} \bg( 1-  \frac{1}{ p^{1+\frac{\tau}{4} } }\bg)}
			\ \, ,\label{eq:ineqB}
		\end{align}
		and, provided $A>0$, \eqref{eq:G3}, \eqref{eq:AB} and \eqref{G^*/G} give
		\begin{align}
							\frac{q(G_{p^{k-1},p^k})}{q(G)} &\geqslant
				\bigg(\frac{\mu(\CE_{p^{k-1},p^k})}{\mu(\CE^*)}\bigg)^{2+\tau} 
				\frac{p^{2+\tau}}{A^{1+\tau}
					\bg(1-  \frac{1}{ p^{1+\frac{\tau}{4} } }\bg)} . \label{eq:ineqC}
		\end{align}
		We now claim that at least one of the following inequalities holds:
		\begin{align}
			\frac{\mu(\CE^+)}{\mu(\CE^*)}	
			&> \bggg( 1- \frac{A}{p}\bggg)^{ \frac{1+\tau}{2+\tau}   }
				\bggg(1- \frac{B}{p}\bggg)^{\frac{1+\tau}{2+\tau}  }
				\bggg(1- \frac{1}{p}
				\bggg)^{ \frac{2}{2+\tau}    }
				\bggg(1-  \frac{1}{ p^{1+\frac{\tau}{4} } }\bggg)^{ \frac{1}{3} } \,; \label{eq:ineq1}\\
			\frac{\mu(\CE_{p^k,p^{k-1}})}{\mu(\CE^*)}
			& >\frac{B^{ \frac{1+\tau}{2+\tau}   }}{p} 
				\bggg(1-  \frac{1}{ p^{1+\frac{\tau}{4} } }\bggg)^{ \frac{1}{3} }\,;\label{eq:ineq2}\\
			\frac{\mu(\CE_{p^{k-1},p^k})}{\mu(\CE^*)}
			& > \frac{A^{ \frac{1+\tau}{2+\tau}   }}{p} 
			\bggg(1-  \frac{1}{ p^{1+\frac{\tau}{4} } }\bggg)^{ \frac{1}{3} }\,.\label{eq:ineq3}
		\end{align}
		If \eqref{eq:ineq1} holds then $q(G^+)\geqslant q(G)$ by \eqref{eq:ineqA}. If \eqref{eq:ineq2} holds, then $\mu(\CE_{p^k,p^{k-1}})>0$, so $B>0$, and so $q(G_{p^k,p^{k-1}})\geqslant q(G)$ by \eqref{eq:ineqB} and \eqref{eq:ineq2}. Finally, if \eqref{eq:ineq3} holds, then $\mu(\CE_{p^{k-1},p^k})>0$, so $A>0$, and so $q(G_{p^{k-1},p^k})\geqslant q(G)$ by \eqref{eq:ineqC} and \eqref{eq:ineq3}. Therefore this claim would complete the proof by choosing $G'\in\{G^+, G_{p^k,p^{k-1}}, G_{p^{k-1},p^k}\}$ according to which of the inequalities \eqref{eq:ineq1}-\eqref{eq:ineq3} hold.
		
		Since $\mu(\CE^+)+\mu(\CE_{p^k,p^{k-1}})+\mu(\CE_{p^{k-1},p^k})=\mu(\CE^*)$, at least one of \eqref{eq:ineq1}-\eqref{eq:ineq3} holds if we can prove that
		\[
		S:= \bggg[
		\bggg(1- \frac{A}{p}\bggg)^{ \frac{1+\tau}{2+\tau}   }
		\bggg(1- \frac{B}{p}\bggg)^{\frac{1+\tau}{2+\tau}  }
		\bggg( 1- \frac{1}{p}\bggg)^{ \frac{2}{2+\tau}    } 
		+ \frac{B^{ \frac{1+\tau}{2+\tau}   }}{p} + \frac{A^{ \frac{1+\tau}{2+\tau}   }}{p} \bggg]
		\bggg(1-  \frac{1}{ p^{1+\frac{\tau}{4} } }\bggg)^{ \frac{1}{3} } < 1.
		\]
		
		Using the inequality $1-x\leqslant e^{-x}$ three times, we find that
		\[
		S\leqslant  \bigg[ 
		\exp\Big(-\frac{(A+B)(1+\tau) +2}{p(2+\tau)} \Big) 
		+ \frac{B^{ \frac{1+\tau}{2+\tau}   }}{p} + \frac{A^{ \frac{1+\tau}{2+\tau}   }}{p} \bigg] 
		\bggg(1-  \frac{1}{ p^{1+\frac{\tau}{4} } }\bggg)^{ \frac{1}{3} } .	
		\]
		
		Since we also have that $e^{-x}\leqslant 1-x+x^2/2$ for $x\geqslant0$, as well as $0\leqslant A,B\leqslant C_2$ and $0 < \tau < 1/100$, we conclude that
		\[
		S\leqslant  
		\bigg( 1-\frac{(A+B)(1+\tau) +2}{p(2+\tau)}+ \frac{C_2^2}{p^2}
		+ \frac{B^{ \frac{1+\tau}{2+\tau}   }}{p} + \frac{A^{ \frac{1+\tau}{2+\tau}   }}{p}\bigg) 
		\bggg(1-  \frac{1}{ p^{1+\frac{\tau}{4} } }\bggg)^{ \frac{1}{3} } .	
		\]
		
		Using the inequality $\frac{x(1+\tau)+1}{2+\tau}\geqslant x^{\frac{1+\tau}{2+\tau}}$ for $x\geqslant 0$ and $\tau >0$ , we obtain that
		\[
		S\leqslant  \bggg( 1+ \frac{C_2^2}{p^2}\bggg)
		\bggg( 1-  \frac{1}{ p^{1+\frac{\tau}{4} } }\bggg)^{ \frac{1}{3} } .	
		\]

		Note that $C_2^2/p^2 \leqslant 1/(3p^{1+\tau/4}) $ for $p \geqslant C_6$  (recall \eqref{eq:CDefs}). Since $(1-x)^{1/3}\leqslant 1-x/3$ for $x\in[0,1]$, we must have that $S \leqslant 1 - 1/(9p^{2+\tau/2}) < 1$ for $p\geqslant C_6$, thus completing the proof of the lemma.
	\end{proof}

	\begin{proof}[Proof of Proposition \ref{prop:IterationStep2}]
		This follows almost immediately from Lemma \ref{lem:MainLem2}. Our assumptions that $\CR(G)\subseteq\{p> C_6\}$ and $\CR^\sharp(G)\neq\emptyset$ imply that there is a prime $p > C_6$ lying in $\CR(G)$. 
		Thus we can apply Lemma \ref{lem:MainLem2} with this choice of $p$ and complete the proof.
	\end{proof}

	\section{Optimality of the exponent $1/2$ in Theorem \ref{thm:quantitative DS}}\label{sec:optimal}
	
	We shall prove the following result, which is sufficient to prove our claim that the exponent $1/2$ in Theorem \ref{thm:quantitative DS} cannot be improved in full generality. 
	
	\begin{proposition}
		Let $\psi(q)=\un_{q\ \text{prime}}/q$. Then
		\[
		\int_0^1\bg( N(\alpha;Q)-\Psi(Q)\bg)^2\dee\alpha= \Psi(Q)+O(1)  .
		\]
	\end{proposition}
	
	\begin{proof} Arguing as in the proof of Theorem \ref{thm:variance} in Section \ref{sec:three propositions}, we have
		\begin{align*}
		\int_0^1\bg( N(\alpha;Q)-\Psi(Q)\bg)^2\dee\alpha
			&= \mathop{\sum\sum}_{q,r\le Q} \lambda(\CA_q\cap\CA_r) - \Psi(Q)^2 \\
			&= \Psi(Q) + \mathop{\sum\sum}_{\substack{q,r\le Q \\ q\neq r}} \lambda(\CA_q\cap\CA_r) - \Psi(Q)^2 .
		\end{align*}
		
		Let $q\neq r$ be two primes. We shall compute $\lambda(\CA_q\cap\CA_r)$ by adapting the proof of Lemma \ref{lem:ovelap estimate}. In the notation there, we have $\ell=m=1$ and $n=qr$. Hence, 
		\[
		\lambda(\CA_q\cap\CA_r) = 2\sum_{\substack{c\ge 1 \\ (c,qr)=1}} w\bgg(\frac{c}{qr}\bgg) ,
		\]
		where $w$ is defined as in Lemma \ref{lem:ovelap estimate}. Recall also the definition of $\delta$ and $\Delta$ there. With our definition of $\psi$, we have $\delta=1/\max\{q,r\}^2$ and $\Delta=1/\min\{q,r\}^2$. 
	By M\"obius inversion and estimate \eqref{eq:w partial summation}, we have
		\begin{align*}
		\lambda(\CA_q\cap\CA_r) 
			= 2\sum_{\substack{c\ge 1 \\ (c,qr)=1}} w\bgg(\frac{c}{qr}\bgg)  
			&= 2\sum_{d|qr}\mu(d) \sum_{c'\ge 1} w\bgg(\frac{c'd}{qr}\bgg)  \\
			&= 4\Delta\delta \phi(q)\phi(r) + O(\delta) \\
			&= \lambda(\CA_q)\lambda(\CA_r) + O\bggg(\frac{1}{\max\{q,r\}^2}\bggg).
		\end{align*}
Hence,
		\begin{align*}
			\int_0^1\bg( N(\alpha;Q)-\Psi(Q)\bg)^2\dee\alpha
			&= \Psi(Q) + \mathop{\sum\sum}_{\substack{q,r\le Q \\ q\neq r}} \lambda(\CA_q)\lambda(\CA_r) - \Psi(Q)^2 
			+ O\bggg( \mathop{\sum\sum}_{\substack{q,r \,\text{prime} \\ r<q\le Q}}   \frac{1}{q^2}\bggg) \\
			&= \Psi(Q) - \sum_{q\le Q}\lambda(\CA_q)^2 + O\bggg( \mathop{\sum\sum}_{\substack{q,r \,\text{prime} \\ r<q\le Q}}   \frac{1}{q^2}\bggg).
		\end{align*}
		The sum of $\lambda(\CA_q)^2$ is uniformly bounded by the choice of $\psi$. In addition, by the prime number theorem, the sum inside the big-Oh is $\ll\sum_{q \, \text{prime} } 1/(q\log q)\ll 1$.  We thus conclude that
		\[
			\int_0^1\bg( N(\alpha;Q)-\Psi(Q)\bg)^2\dee\alpha= \Psi(Q)+O(1),
		\]
		as claimed. This completes the proof of the proposition. 
	\end{proof}

	\section*{Acknowledgements}
	The authors would like to thank Manuel Hauke, Santiago Vazquez Saez and Aled Walker for sharing a preliminary version of their preprint \cite{HVW} with them. D.Y.~would also like to thank Christoph Aistleitner for introducing him to this problem and for some interesting discussion on the reference \cite{ABH}. Finally, we thank the anonymous referees of the paper for their careful reading and valuable comments.
	
	\medskip
	
	D.K.~was supported by the Courtois Chair II in fundamental research, by the Natural Sciences and Engineering Research Council of Canada (RGPIN-2018-05699) and by the Fonds de recherche du Qu\'ebec - Nature et technologies (2022-PR-300951).
	
	J.M.~was supported  by the European Research Council (ERC) under the European Union’s Horizon 2020 research and innovation programme (grant agreement No 851318).
	
	 D.Y.~was supported  by a postdoctoral fellowship funded by the Courtois Chair II in fundamental research,
 the Natural Sciences and Engineering Research Council of Canada, and the Fonds de recherche du Qu\'ebec - Nature et technologies.

\end{document}